\documentclass[11pt, a4paper]{article}

\usepackage[a4paper, nomarginpar]{geometry}	
\usepackage[dvipsnames]{xcolor}
\usepackage[utf8]{inputenc}
\usepackage[english]{babel}
\usepackage{csquotes}
\usepackage{amsmath}
\usepackage{amsthm}
\usepackage{amssymb}
\usepackage{amsfonts}
\usepackage{mathtools}
\usepackage{mathrsfs}
\usepackage{tikz-cd}
\usepackage{layout}
\usepackage{framed}
\usepackage{graphicx}
\usepackage{tensor}
\usepackage{esvect}
\usepackage[hidelinks]{hyperref}
\usepackage{pdfpages}
\usepackage{tocloft}

\usepackage[
  backend=biber,
  style=numeric, 
  sorting=nyt
]{biblatex}

\colorlet{shadecolor}{gray!25}   
\renewenvironment{leftbar}{%
  \MakeFramed {\advance\hsize-\width \FrameRestore}}%
{\endMakeFramed}

\theoremstyle{plain}
\newtheorem{lma}{Lemma}[section]
\newtheorem{lemma}[lma]{Lemma}
\newtheorem{prop}[lma]{Proposition}

\newtheorem{theorem}[lma]{Theorem}

\newtheorem{corollary}[lma]{Corollary}
\theoremstyle{definition}
\newtheorem{definition}[lma]{Definition}
\newtheorem{exampleus}[lma]{Example}
\newenvironment{example}[1][]
    {\begin{leftbar} \vspace{-9pt} \begin{exampleus}[#1]}
    {\end{exampleus} \vspace{-9pt}\end{leftbar}}

\newtheorem{remark}[lma]{Remark}

\newcommand{\N}{\mathbb{N}}
\newcommand{\Z}{\mathbb{Z}}

\newcommand{\R}{\mathbb{R}}

\newcommand{\scrF}{\mathscr{F}}

\newcommand{\scrX}{\mathscr{X}}

\newcommand{\scrD}{\mathscr{D}}

\newcommand{\calD}{\mathcal{D}}

\newcommand{\calM}{\mathcal{M}}
\newcommand{\calN}{\mathcal{N}}

\newcommand{\calR}{\mathcal{R}}

\newcommand{\calU}{\mathcal{U}}

\newcommand{\rmT}{\mathrm{T}}

\newcommand{\pr}{\mathrm{pr}}

\newcommand{\Mod}{\mathsf{Mod}}

\newcommand{\cifty}{C^\infty}

\newcommand{\supp}{\mathrm{supp}\,}

\newcommand{\id}{\mathrm{id}}
\newcommand{\im}{\mathrm{im}}

\newcommand{\op}{\mathrm{op}}

\newcommand{\rest}[2]{\left.#1\right|_{#2}}
\newcommand{\parest}[2]{\rest{\left(#1\right)}{#2}}
\newcommand{\dd}[2]{\frac{\partial #1}{\partial #2}}

\newcommand{\bld}[1]{\boldsymbol{#1}}
\newcommand{\uline}[1]{\underline{#1}}

\definecolor{linkcol}{HTML}{1C549E}

\hypersetup{colorlinks=true, allcolors=linkcol}

\title{Graded Vector Bundles}
\author{JV, RS}

\begin{document}

\begin{flushright}
\today
\end{flushright}
\vspace{0.7cm}
\begin{center}

\baselineskip=13pt {\Large \bf{On the Integrability of Distributions in $\mathbb{Z}$-graded Geometry}\\}
 \vskip0.5cm
 {\large{Rudolf \v{S}molka$^{1}$, Jan Vysoký$^{1}$}}\\
 \vskip0.6cm
$^{1}$\textit{Faculty of Nuclear Sciences and Physical Engineering, Czech Technical University in Prague\\ Břehová 7, 115 19 Prague 1, Czech Republic, rudolf.smolka@cvut.cz}\\ 
\vskip0.3cm
\end{center}

\begin{abstract}
Integrability of distributions on $\Z$-graded manifolds is examined. First, the Local Frobenius theorem -- the local existence of flat coordinates for an involutive distribution -- is proved. Then, two different notions of an integral submanifold are discussed. It is found that, for the stronger of the two notions, the Global Frobenius theorem holds, but as one implication only: involutivity implies integrability. It is then shown that integrability implies a certain weaker version of involutivity. Simple counterexamples of the converse implications are given.
\end{abstract}


\vfill

\begingroup
  \hypersetup{hidelinks}
  \tableofcontents
\endgroup

\vfill

\newpage

\section{Introduction and Summary}

One of the cornerstones of differential geometry and the theory of differential equations is the Frobenius theorem, which characterizes when a smooth distribution on a manifold admits integral submanifolds. Its origins lie in the classical problem of determining when a system of first-order partial differential equations is completely integrable, namely in the 19th century works of Deahna \cite{Deahna1840}, Clebsch \cite{Clebsch1866} and Frobenius \cite{Frobenius1877}. In today's parlance, the problem may be stated as follows: given a (smooth, regular) distribution on a smooth manifold, does there, through every point, pass a submanifold whose tangent spaces span the fibers of the distribution? The powerful and elegant answer is given by what we now call the Frobenius theorem -- it is both sufficient and necessary for the distribution to be involutive, i.e. closed under the commutator of vector fields.

The past decades have seen a surge of interest in $\N$-, $\Z$- or $\Z_2^n$-graded manifolds, on top of $\Z_2$-graded manifolds (supermanifolds) so thoroughly studied in the past century. The Frobenius theorem has naturally been one of the subjects of exploration, though a $\Z$-graded investigation is, to our knowledge, still missing. The literature is naturally richest in the supermanifold case, where the Frobenius theorem appeared already in the 1984 short paper \cite{Bruzzo1985}. Among other works, let us point out the 1997 paper \cite{Monterde1997}, which closely resembles the text presented here, both in content and in form. To mention more recent works, the $\N$-graded case of (what we call here) the Local Frobenius theorem appears in the 2024 paper \cite{NFT}. For $\Z_2^n$-graded manifolds, both the Local and Global versions appear in the 2016 preprint \cite{Z2nFT}, though the version of the Global Frobenius theorem there seems to be in conflict with the one in \cite{Monterde1997}; the authors of the former obtain a stronger result with a weaker notion of an integral submanifold.

This paper is not intended to be self-contained inasmuch as some knowledge of $\Z$-manifold theory is necessary to fully follow some of the arguments within. See \cite{GTGM} for a comprehensive introduction into $\Z$-manifolds. We also build on some results from our previous works, namely \cite{3GVB} and \cite{Zflows}. For this reason, we provide only a very brief recollection of notation in Section \ref{section_2}. That being said, we believe that only a passing knowledge of, or interest in, graded manifolds of any kind is necessary to absorb the main ideas and results of the paper. Let us stress that we treat only regular, not singular, distributions. The singular case is much more involved already in the smooth case (\cite{Androulidakis2009}, \cite{CrainicFernandes2011}), and we are not aware of any graded ventures in this direction.

The main body of the text is composed of technical statements accompanied by detailed proofs. Thus, we believe it useful to provide a short summary of the text, along with a highlight of the main results.

We begin \textbf{Section 2 -- Distributions and the Local Frobenius Theorem} with a reminder of notation and conventions. The goal of this section is to show the existence of flat coordinates for any involutive distribution -- the Local Frobenius theorem (Theorem \ref{thm_local_frob}). This is done by mimicking the work done for other graded manifolds, namely \cite{Z2nFT} and \cite{NFT}. As was so keenly observed in \cite{Monterde1997}, the proof of Theorem \ref{thm_local_frob} consists of two parts. The first is to show that any involutive distribution locally admits a frame of commuting vector fields, which is done algebraically and without much difficulty. The second part is to find coordinates, in which such frame is given by the coordinate vector fields: the so-called flat coordinates for the distribution. This task is, in essence, analytic -- one needs to solve a (possibly infinite) set of partial differential equations. Here we take advantage of our previous work on flows on $\Z$-manifolds \cite{Zflows}.

In \textbf{Section 3 -- Tangent Vector Fields and Integral Submanifolds}, we begin collecting observations that will eventually lead to the Global Frobenius theorem (Theorem \ref{thm_global_frob}). In particular, we seek to identify a suitable notion of an integral submanifold for a distribution. We present two non-equivalent possibilities (Definition \ref{def_integral_submanifold_1}). The weaker one, which we choose to call ``fiberwise integral'' compares only tangent spaces with fibers of the distribution. This notion was used to define integral submanifolds in \cite{Z2nFT}. The second concept, to which we assign the name “integral,” is not restricted to points; instead, one compares the module of vector fields on the submanifold with the vector fields in the distribution, making use of the notion of a tangent vector field (Definition \ref{def_tangent_vector_field}). We say that a distribution $\scrD$ on $\calM$ is integrable, if every point of $M$ is contained within (the image of) an integral submanifold for $\scrD$.

We proceed with an alternative way to define fiberwise integral and integral submanifolds using the language of pullback bundles (Proposition \ref{prop_equivalently_integral}). To do this, we recall some relevant results from the theory of sheaves, graded vector bundles, and category theory. We finish the section by giving an example of an involutive distribution that admits both integral and (merely) fiberwise integral submanifolds, Example \ref{example_strongly_and_weakly_integral_submanifolds}.


In \textbf{Section 4 -- Global Frobenius Theorem}, we start by an obvious notion of equivalence for submanifolds (Definition \ref{def_equivalent_submanifolds}).
We show that, up to this equivalence, two integral submanifolds of an involutive distribution must be the same (Proposition \ref{prop_similarly_immersed_strongly_integral}), and that this is not true for fiberwise integral submanifolds (Corollary \ref{cor_nonequivalently_integral}). We visit the example of level-sets of a submersion (Example \ref{example_kernel_of_sumbersion}). The highlight of the section is the Global Frobenius Theorem (Theorem \ref{thm_global_frob}) in which we show that on a $\Z$-manifold with an involutive distribution, every point is contained within a unique (up to the discussed equivalence) maximal connected integral submanifold. In this sense, we prove involutivity to be a sufficient, but not necessary, condition for integrability. We would like to contrast this result with the corresponding result in \cite[Theorem 5.6]{Monterde1997} for supermanifolds, where the authors identify an additional condition on the foliation of the underlying smooth manifold.

In \textbf{Section 5 -- Pointwise Involutivity and the Converse Implication}, we begin with two examples (Example \ref{example_integrable_noninvolutive_distr}, Example \ref{example_integrable_noninvolutive_distr2}) which demonstrate that involutivity of a distribution is indeed only a sufficient condition for integrability. This is in contrast with the $\Z_2^n$-graded result in \cite[Theorem 7.2]{Z2nFT}, where the result reached is an ``if and only if'' type of statement. We then identify a weaker notion of involutivity, the so-called ``pointwise involutivity'' (Definition \ref{def_pointwise_involutive}). In Theorem \ref{thm_on_foliated_distributions} we find that an integrable distribution on a $\Z$-manifold $\calM$ must be pointwise involutive. However, a pointwise involutive distribution need not be integrable, as demonstrated in Example \ref{example_pointwise_involutive_nonintegrable}.

\section*{Acknowledgments}

The research was supported by grant GA\v{C}R 24-10031K. The authors are also grateful for financial support from M\v{S}MT under grant no. RVO 14000 and from the Grant Agency of the Czech Technical University in Prague under grant no. SGS25/163/OHK4/3T/14.

\section{Distributions and the Local Frobenius Theorem}\label{section_2}

Let us start with a review of notation. For a concise recollection of the topic of $\Z$-manifolds, the reader is encouraged to see either of our previous works \cite{3GVB}, \cite{Zflows}. For a thorough introduction to $\Z$-graded geometry, see \cite{GTGM}.

A $\Z$-graded manifold, or a $\Z$-manifold for short, is given as a pair $\calM = (M, \cifty_\calM)$, where $M$ is the underlying smooth manifold, and $\cifty_\calM$ is the structure sheaf. We denote as $\op(M)$ the set of all open subsets of $M$, and as $\op_p(M)$ the set of all open neighborhoods of a point $p$ in $M$. For any $U \in \op(M)$, the elements $f \in \cifty_\calM(U)$ are called graded functions and they each carry an integer degree $|f|\in \Z$. Their commutativity is determined by the parity of the degrees $f g = (-1)^{|f| |g|} \,g f $, where $|fg| = |f| + |g|$. 

The local model for $\Z$-manifolds is $U^{(m_j)} := (U, \cifty_{(m_j)})$, where $(m_j)\equiv (m_j)_{j \in \Z}$ is a sequence of non-negative integers, only finitely many of which are non-zero, $U$ is an open subset of $\R^{m_0}$, and $\cifty_{(m_j)}$ assigns to every $V \in \op(U)$ the graded algebra $\cifty_{(m_j)}(V)$ of formal power series
\begin{equation}
    f = \sum_{\bld{p}} f_{\bld{p}} \, (\xi^{1})^{p_1} \cdots (\xi^{\hat{m}})^{p_{\hat{m}}}.
\end{equation}
Here $f_{\bld{p}}$ are smooth functions on $V$ carrying degree zero, and $\{\xi^a\}_{a = 1}^{\hat{m}}$ are formal variables, also called graded coordinates, each carrying a non-zero degree. The sum ranges over multi-indices $\bld{p}$ such that each term in the series has the same degree -- only homogeneous elements are allowed by construction, see \cite{GTGM}. Above we wrote $\hat{m} := \sum_{j \neq 0}m_j$. 

If $\calM$ is locally isomorphic to $U^{(m_j)}$, we say that $(m_j)$ is the graded dimension of $\calM$. If $\{x^\mu\}_{\mu = 1}^{m_0}$ are the standard coordinates on $U \subseteq \R^{m_0}$, we say that $\{x^\mu\}_{\mu = 1}^{m_0}$ together with $\{\xi^a\}_{a = 1}^{\hat{m}}$ are coordinates on $\calM|_{U}$. We often denote them together as $\{x^i\}_{i = 1}^{m}$, where $m := m_0 + \hat{m} = \sum_{j \in \Z} m_j$.

Morphisms in the category of $\Z$-manifolds, or graded smooth maps, are also pairs $\varphi \equiv (\uline{\varphi}, \varphi^\ast) : \calM \to \calN$, where $\uline{\varphi} : M \to N$ is a smooth map and $\varphi^\ast : \cifty_\calN \to \uline{\varphi}_\ast \cifty_\calM$ is a morphism of sheaves of graded algebras. Vector fields on $\calM$ are defined as graded derivations of graded functions, and their sheaf is denoted as $\scrX_\calM$. They form a sheaf of $\cifty_\calM$-modules. We will often drop the word ``graded'' to keep its proliferation in the text manageable. Vector fields on $\calM$ always admit a local frame, namely the coordinate local frame $\{\partial_{x^i}\}_{i = 1}^{m}$, with degrees $|\partial_{x^i}| = - |x^i|$. We use the word ``frame'' to refer to free generators of a module. We keep the word ``basis'' to refer to free generators of a vector space. As the sheaf of $\cifty_\calM$-modules $\scrX_\calM$ is locally freely and finitely generated, it forms the sheaf of sections of a graded vector bundle $T \calM$. Since $|\partial_{x^i}| = - |x^i|$, the rank of the tangent bundle $T \calM$ is $(m_{-j})$, meaning there are, locally, $m_j$ free generators of degree $-j$.

Now onto the substance of the paper.

\begin{definition}
Let $\calM$ be a graded manifold. Then any subbundle $\scrD$ of the tangent bundle $T\calM$ is called a \textbf{distribution} on $\calM$. We will use the same letter $\scrD$ to refer to its sheaf of sections. If $M$ is not connected, we also require $\scrD$ to have constant rank (this is automatically satisfied if $M$ is connected). A distribution $\scrD$ is called \textbf{involutive} if $\scrD(U)$ is closed under the graded commutator $[X,Y] := X \circ Y - (-1)^{|X||Y|}\, Y \circ X$ for every $U \in \op(M)$.
\end{definition}

\begin{lemma}\label{lemma_linearly_indep_vectors}
    Let $\scrD$ be a distribution on $\calM$, and let $X_1, \dotsc, X_r$ be a local frame for $\scrD$ on some $U \in \op(M)$. Then $X_1|_{p}, \dotsc, X_r|_{p}$ are linearly independent vectors in $T_p\calM$.
\end{lemma}
\begin{proof}
    Consider some $p \in U$. As $\scrD$ is a subbundle of $\scrX_\calM$, there exists a local frame $Y_1, \dotsc, Y_m$ for $\scrX_\calM$ on some $V \in \op(U)$ such that $Y_1, \dotsc, Y_r$ is a local frame for $\scrD$, see \cite{3GVB}. Hence, $X_1|_{V}, \dotsc, X_r|_{V}, $ together with $Y_{r+1}, \dotsc, Y_m$ also form a local frame for $\scrX_\calM$. This can be verified directly from the definition of a frame. Consequently, they are linearly independent at every point and the result follows. Note that even in non-graded geometry it is not sufficient that $\scrD$ be a sheaf of submodules of $\scrX_\calM$. 
\end{proof}

The following lemma makes certain we can verify involutivity of a distribution both locally and globally.

\begin{lemma} \label{lemma_local_to_global_involutivity}
Let $\scrD$ be a distribution on $\calM$. Then the following are equivalent:
\begin{enumerate}
\item $\scrD$ is involutive.
\item $\scrD(M)$ is closed under the graded commutator.
\item For every $p \in M$ there exists $U \in \op_p(M)$ such that $\scrD(U)$ is closed under the graded commutator.
\item For every $p \in M$ there exists $U \in \op_p(M)$ together with a frame $X_1, \dotsc, X_r$ for $\scrD(U)$ such that $[X_i, X_j] \in \scrD(U)$ for all $i,j \in \{1, \dotsc, r\}$.
\end{enumerate}
\end{lemma}
\begin{proof}
    The chain of implications $\mathit{1.} \implies \mathit{2.} \implies \mathit{3.}$ is trivial. To show $\mathit{3.} \implies \mathit{4.}$ take some $p \in M$ and $U \in \op_p(M)$ promised by point $\mathit{3.}$ Find some $V \in \op_p(U)$ such that there exists a frame $X_1, \dotsc, X_r$ for $\scrD(V)$. We will show that $[X_i, X_j] \in \scrD(V)$ for all $i,j \in \{1, \dotsc, m\}$ using the usual smooth bump function argument. Take some arbitrary $q \in V$, some $W \in \op_q(V)$ such that $\overline{W} \subseteq V$, where $\overline{W}$ is the closure of $W$. Consider a smooth bump function $\lambda \in \cifty_\calM(U)$ such that $\lambda|_{W} = 1$ and $\supp(\lambda) \subseteq V$. For the existence of graded smooth bump functions and partitions of unity, see \cite[Section 3.5]{GTGM}. Then extend each $X_i$ to $\tilde{X}_i \in \scrX_\calM(U)$ from $W$ by zero. Explicitly, set $\tilde{X}_i|_{V} =  \lambda|_{V} \, X_i$, $\tilde{X}_i|_{U \setminus \supp(\lambda)} = 0$. As $\scrD$ is a sheaf, there is $\tilde{X}_i \in \scrD(U)$ for every $i$. Hence, by assumption, $[\tilde{X}_i, \tilde{X}_j] \in \scrD(U)$ for all $i,j \in \{1, \dotsc, m\}$. It follows that
    \begin{equation}
           [X_i, X_j]|_{W} = [X_i|_{W}, X_j|_{W}] = [\tilde{X}_i|_{W}, \tilde{X}_j|_{W}] = [\tilde{X}_i, \tilde{X}_j]|_{W} \in \scrD(W).
    \end{equation}
    As $q \in V$ was arbitrary, and $\scrD$ is a sheaf, there must be $[X_i, X_j] \in \scrD(V)$ for any $i,j \in \{1, \dotsc, r\}$.
    
    To show $\mathit{4.}\implies \mathit{1.}$ take some $U \in \op(M)$. Let $p \in U$ be arbitrary, and let $X_1, \dotsc, X_r$ be the frame for $\scrD(V)$ for some $V \in \op_p(M)$ promised by point $\mathit4.$ Denote $W := U \cap V$, and note that $X_1|_{W}, \dotsc, X_r|_{W}$ is a frame for $\scrD(W)$. The result then follows from the fact that $[X_i|_{W}, X_j|_{W}] = [X_i, X_j]|_{W}$, Leibniz rule of the commutator, and the fact that $\scrD$ is a sheaf.
\end{proof}

\begin{definition}
    Let $\scrD$ be a  distribution on $\calM$. Let $U \in \op(M)$ with coordinates $\{x^\mu\}_{\mu = 1}^{m_0}$ and $\{\xi^a\}_{a = 1}^{\hat{m}}$ on $\rest{\calM}{U}$. We say that $x^\mu, \xi^a$ are \textbf{flat coordinates for $\scrD$} if $\{\partial_{x^\mu}\}_{\mu = 1}^{r_0}$ and $\{\partial_{\xi^a}\}_{a = 1}^{\hat{r}}$ together form a frame for $\scrD(U)$.
\end{definition}

In the above definition we distinguish degree zero and nonzero coordinates. However, as mentioned before, we will often not make this distinction and denote all local coordinates on $\calM$ together as $\{x^i\}_{i = 1}^m$. Saying that these are flat coordinates for $\scrD$ will mean that $\{\partial_{x^i}\}_{i = 1}^r$ form a frame for $\scrD(U)$. 

The local Frobenius theorem is simply the statement that, for every involutive distribution, there exist flat coordinates for the distribution around every point. We formally state and prove it at the end of the section. To do so, we start by showing that an involutive distribution locally always admits a frame of mutually commuting vector fields, see \cite[Theorem 3.4]{Monterde1997}, \cite[Proposition 5.2]{Z2nFT} or \cite[Theorem 6.4]{NFT}.

\begin{prop}\label{prop_commuting_frame}
Let $\scrD$ be a distribution on $\calM$. Then $\scrD$ is involutive if and only if every point $p \in M$ has a neighborhood, on which $\scrD$ has a local frame $Y_1, \dotsc, Y_r$ such that $[Y_i, Y_j] = 0$ for all $i,j \in \{1, \dotsc, r\}$.
\end{prop}
\begin{proof}
The proof is essentially a $\Z$-graded copy of the proof in \cite{Z2nFT}. The ``if'' direction follows immediately, so consider some involutive distribution $\scrD$. As the statement is local in nature, we may, without loss of generality, consider $\calM$ to be a graded domain $\calM = U^{(m_j)}$ for some $U \subseteq \R^{m_0}$. We may assume as well the existence of a global frame $\{X_{i}\}_{i = 1}^r$ for $\scrD$. Denote the global coordinates on $U^{(m_j)}$ as $\{x^i\}_{i = 1}^m$.

By Lemma \ref{lemma_linearly_indep_vectors} the vectors $\{X_i|_p\}_{i = 1}^r$ are linearly independent at every point $p \in U$. By shrinking $U$ and relabeling $\{x^i\}$ if necessary, we may thus assume that $X_1, \dotsc, X_r, \partial_{x^{r+1}}, \dotsc, \partial_{x^m}$ form a global frame for all vector fields on $U^{(m_j)}$. Note that we have another global frame for vector fields on $U^{(m_j)}$, namely the coordinate one $\partial_{x^1}, \dotsc, \partial_{x^m}$. As a result, there must exist graded functions ${a_i}^j$, ${\omega_i}^j$, ${b_i}^\ell$ and ${\lambda_i}^\ell$ for all $i,j \in \{1, \dotsc, r\}$ and $\ell \in \{r+1, \dotsc, m\}$ such that
\begin{equation}
    X_i = \sum_{j = 1}^r {a_i}^j \partial_{x^j} + \sum_{\ell = r+1}^m {b_i}^\ell \partial_{x^\ell}, \qquad
    \partial_{x^i} = \sum_{j = 1}^r {\omega_i}^j X_j + \sum_{\ell = r+1}^m {\lambda_i}^\ell \partial_{x^\ell},
\end{equation}
for any $i \in \{1, \dotsc, r\}$. By combining these expressions we find that for any $i \in \{1, \dotsc, r\}$ there is
\begin{equation}
    \partial_{x^i} = \sum_{j = 1}^r {\omega_i}^j \left(\sum_{k = 1}^r {a_j}^k \partial_{x^k} + \sum_{\ell = r+1}^m {b_j}^\ell \partial_{x^\ell} \right) + \sum_{\ell = r+1}^m {\lambda_i}^\ell \partial_{x^\ell},
\end{equation}
hence, in particular, $\sum_{j = 1}^r {\omega_i}^j \, {a_j}^k = {\delta_i}^k$. For any $i \in \{1, \dotsc, r\}$ denote
\begin{equation}
    Y_i := \sum_{j = 1}^r  {\omega_i}^j X_j = \sum_{j = 1}^r  {\omega_i}^j \left(\sum_{k = 1}^r {a_j}^k \partial_{x^k} + \sum_{\ell = r+1}^m {b_j}^\ell \partial_{x^\ell} \right) = \partial_{x^i} + \sum_{j = 1}^r\sum_{\ell = r+1}^m {\omega_i}^j {b_j}^\ell \partial_{x^\ell}.
\end{equation}
Note two things: $Y_1, \dotsc, Y_r$ form a frame for $\scrD(U)$, and every $Y_i$ is of the form $Y_i = \partial_{x^i} + \sum_{\ell = r+1}^m \, {f_i}^\ell \partial_{x^\ell}$ for some graded functions ${f_i}^\ell \in \cifty_{(m_j)}(U)$. The reader can easily verify that, as a result, also the commutators are of the form
\begin{equation}
    [Y_i, Y_j] = \sum_{\ell = r+1}^m {g_{ij}}^\ell \, \partial_{x^\ell},
\end{equation}
for some ${g_{ij}}^\ell \in \cifty_{(m_j)}(U)$. Now the involutivity of $\scrD$ comes into play: it dictates that ${g_{ij}}^\ell = 0$ for every $i,j \in \{1, \dotsc, r\}$ and $\ell \in \{r+1, \dotsc, m\}$. In other words, $[Y_i, Y_j] = 0$ and $Y_1, \dotsc, Y_r$ is a frame for $\scrD(U)$ formed by commuting vector fields, as desired.

\end{proof}

Now we need to recall a few facts about flows of graded vector fields. We sum them up in the form of a theorem; for more on flows on $\Z$-manifolds, the reader is referred to \cite{Zflows}. 

Recall that a \textbf{flow domain} of degree $k \in \Z$ is the graded manifold $\calM \times \R[-k]$ if $k \neq 0$, or the graded manifold $(\calM \times \R)|_{D}$ where $D$ is a flow domain on $M$, if $k = 0$. A flow domain $D$ on $M$ is simply any open subset $D \subseteq M \times R$ such that $\{t \in \R \, | \, (p,t) \in D\}$ is an open interval containing $0$ for every $p \in M$.

\begin{theorem}[Some Properties of Flows]\label{thm_flows_summary}
    Let $X$ be a global vector field on $\calM$, such that $[X,X] = 0$. Then there exists a flow $\theta_X$ of $X$. In particular, this is a graded smooth map $\theta_{X} : \calD \to \calM$, where $\calD$ is a flow domain of degree $|X|$, with the following properties:
    \begin{enumerate}
        \item $1 \otimes \partial_{t} \sim_{\theta_X} X$, where $t$ is the coordinate on $\R[-k]$.
        \item $\theta_{X}$ is a surjective submersion.
        \item For any $Y \in \scrX_\calM(M)$ there is $Y \otimes 1 \sim_{\theta_X} Y$ if and only if $[X,Y] = 0$.
    \end{enumerate}
\end{theorem}
\begin{proof}
    Existence of $\theta_X$ and property $\mathit{1.}$ follow from the Fundamental theorem on flows on $\Z$-graded manifolds, \cite[Theorem 3.10]{Zflows}. Property $\mathit{2.}$ follows from the fact that $\theta_X$ has a right inverse, \cite[Definition 3.2]{Zflows} and property $\mathit{3.}$ is the subject of \cite[Theorem 4.4]{Zflows}.
\end{proof}

We continue with a crucial but somewhat technical lemma, which will be used inductively to prove the local Frobenius theorem.

\begin{lemma}\label{lemma_local_frobenius}
Let $U \subseteq \R^{m_0}$ be an open set containing $\bld{0} \in \R^{m_0}$, and let $U^{(m_j)}$ be a graded domain with coordinates $\{x^i\}_{i = 1}^m$. Let $r \in \{1, \dotsc, m\}$ and let $X$ be a vector field satisfying $[X,X] = 0$ and $[X, \partial_{x^j}] = 0$ for all $j \in \{1, \dotsc, r-1\}$. Furthermore, assume that the tangent vectors $\rest{\partial_{x^1}}{\bld{0}}, \dotsc, \rest{\partial_{x^{r-1}}}{\bld{0}}, \rest{X}{\bld{0}}$ are linearly independent. Then there exists $U^\prime \in \op_{\bld{0}}(U)$, $V \in \op_{\bld{0}}(\R^{m_0})$ and a  graded diffeomorphism $\chi : (U^\prime)^{(m_j)} \to V^{(m_j)}$ satisfying $\uline{\chi}(\bld{0}) = \bld{0}$, which introduces coordinates $\{y^i\}_{i = 1}^m$ on $(U^\prime)^{(m_j)}$ such that $\rest{\partial_{x^j}}{U^\prime} = \partial_{y^j}$ for all $j \in \{1, \dotsc, r-1\}$ and $\rest{X}{U^\prime} = \partial_{y^{r}}$.
\end{lemma}

\begin{proof}
    The main idea of the proof is the same for any degree of $|X|$, but for technical reasons we will divide the proof into two cases.

    \textbf{Case 1:} $|X| = 0$. Let $\theta_X : \calD \to U^{(m_j)}$ be a flow of $X$ from Theorem \ref{thm_flows_summary}. Note that $\uline{\theta_X}$ is a flow of the underlying vector field $\uline{X}\in \scrX(U)$, see \cite{Zflows}. Also note that $\calD = (U^{(m_j)}\times \R)|_D$ where $D$ is in particular an open subset of $U \times \R$ containing $(\bld{0},0) \in U \times \R$. As a result, there exists an open interval $I \subseteq \R$ containing $0 \in \R$, such that the open cube $Q := I^{\times (m_0+1)} \subseteq \R^{m_0 + 1}$ satisfies $Q \subseteq D$. From now on we shall consider $\theta_X$ restricted to the map $\theta_X : \calD|_Q \to U^{(m_j)}$. Notice that
    \begin{equation}
        \calD|_{Q} = Q^{(\tilde{m}_j)},
    \end{equation}
    is itself a graded domain, where $\tilde{m}_j = m_j + \delta_{0j}$. In fact, we can write $Q^{(\tilde{m}_j)} = (V \times I)^{(m_j)} = V^{(m_j)}\times I$, where $V := I^{\times m_0}\subseteq U$. Let $\{x^i\}_{i = 1}^{m}$ be coordinates on $U^{(m_j)}$. Then, by slight abuse of notation, we have coordinates $\{x^{i}\}_{i= 1}^m$ and $t$ on $Q^{(\tilde{m}_j)} = V^{(m_j)}\times I$. From Theorem \ref{thm_flows_summary} we know that 
    \begin{equation}
         \partial_{t} \sim_{\theta_X}X,\qquad \text{and}\qquad \partial_{x^i} \sim_{\theta_X}\partial_{x^i},
    \end{equation}
    for all $i \in \{1, \dotsc, r-1\}$. In particular this implies $(T_{(\bld{0},0)}\theta_X)(\rest{\partial_t}{(\bld{0},0)}) = \rest{X}{\bld{0}}$ and $(T_{(\bld{0},0)}\theta_X)(\rest{\partial_{x^i}}{(\bld{0},0)}) = \rest{\partial_{x^i}}{\bld{0}}$ for all $i \in \{1, \dotsc, r-1\}$. From Theorem \ref{thm_flows_summary} we also know that $\theta_X$ is a submersion, i.e. the collection of $m + 1$ tangent vectors
    \begin{equation}
        \rest{\partial_{x^1}}{\bld{0}}, \dotsc, \rest{\partial_{x^{r-1}}}{\bld{0}}, \rest{X}{\bld{0}}, (T_{(\bld{0},0)}\theta_X)(\rest{\partial_{x^r}}{(\bld{0},0)}), \dotsc, (T_{(\bld{0},0)}\theta_X)(\rest{\partial_{x^m}}{(\bld{0},0)}),
    \end{equation}
    must span the tangent space $T_{\bld{0}}U^{(m_j)}$. This collection is linearly dependent, thus one of the vectors has to be expressible as a linear combination of previous ones, in the order as presented. By relabeling the coordinates $x^r, x^{r+1}, \dotsc, x^m$ if necessary, we conclude that the vectors
    \begin{equation}\label{eq_basis_zerodeg}
        \rest{\partial_{x^1}}{\bld{0}}, \dotsc, \rest{\partial_{x^{r-1}}}{\bld{0}}, \rest{X}{\bld{0}}, (T_{(\bld{0},0)}\theta_X)(\rest{\partial_{x^{r+1}}}{(\bld{0},0)}), \dotsc, (T_{(\bld{0},0)}\theta_X)(\rest{\partial_{x^m}}{(\bld{0},0)}),
    \end{equation}
    form a basis for $T_{\bld{0}}U^{(m_j)}$. In particular this implies $|x^r| = |X| = 0$. Next step is to define a graded smooth map $\phi : V^{(m_j)} \to Q^{(\tilde{m}_j)}$ by the pullbacks
    \begin{equation}
        \phi^\ast(t) = x^r, \qquad
        \phi^\ast(x^r) = 0, \qquad
        \phi^\ast(x^i) = x^i, \quad \text{for }i \neq r.
    \end{equation}
    Note that this is possible since $|x^r| = 0$. To write the underlying map $\uline{\phi}$, let us again distinguish between degree zero and nonzero coordinates $x^i \equiv (y^\mu, \xi^a)$. Then there exists $\nu \in \{1, \dotsc, m_0\}$ such that $y^\nu = x^r$, and we find
    \begin{equation}
        \uline{\phi}(q_1, \dotsc, q_{\nu - 1}, q_\nu, q_{\nu+1}, \dotsc, q_{m_0}) = (q_1, \dotsc, q_{\nu - 1}, 0, q_{\nu+1}, \dotsc, q_{m_0}, q_{\nu}) \in Q = V \times I,
    \end{equation}
    for any $(q_1, \dotsc, q_{m_0})\in V$. This only makes sense because $Q$ is a cube. From the definition of $\phi$ we have $\partial_{x^r} \sim_{\phi} \partial_t$ and  $\partial_{x^i} \sim_\phi \partial_{x^i}$ for any $i \neq r$. Consider the composite map $\vartheta : =\theta_X \circ \phi : V^{(m_j)} \to U^{(m_j)}$. We see that
    \begin{equation}\label{eq_switcharoo_zerodeg}
        \partial_{x^r} \sim_{\vartheta} X, \qquad \text{and} \qquad \partial_{x^i} \sim_{\vartheta}\partial_{x^i},
    \end{equation}
    for all $i \in \{1, \dotsc, r-1\}$. Furthermore, for any $i \in \{r+1, \dotsc, m\}$ there is 
    \begin{equation}
        (T_{\bld{0}}\vartheta)(\rest{\partial_{x^i}}{\bld{0}}) = (T_{(\bld{0},0)}\theta_X)(\rest{\partial_{x^i}}{(\bld{0},0)}),
    \end{equation}
    which together with (\ref{eq_switcharoo_zerodeg}) and the fact that (\ref{eq_basis_zerodeg}) is a basis means that $T_{\bld{0}}\vartheta$ is a linear bijection. In other words, $\vartheta : V^{(m_j)} \to U^{(m_j)}$ is a local diffeomorphism at $\bld{0}\in V$. Since $\uline{\vartheta}(\bld{0}) = \bld{0}$, there exist some $V^\prime \in \op_{\bld{0}} V$ and $U^\prime \in \op_{\bld{0}}(U)$ such that
    \begin{equation}
        \rest{\vartheta}{V^{\prime}} : (V^\prime)^{(m_j)} \to (U^\prime)^{(m_j)},
    \end{equation}
    is a diffeomorphism, see \cite[Theorem 4.30]{GTGM}. Its inverse $\rest{\vartheta}{V^{\prime}}^{-1} =: \chi$ is the desired coordinate map on $(U^\prime)^{(m_j)}$.

    \textbf{Case 2:} $|X| \neq 0$. Denote $k := |X|$. Again, consider the flow $\theta_X : \calD \to U^{(m_j)}$ of $X$, where $\calD = U^{(m_j)} \times \R[-k]$, as per Theorem \ref{thm_flows_summary}. Note that the underlying map is the trivial map $\uline{\theta_{X}} : M \times \{\ast\} \to M, \ (m, \ast) \mapsto m$, see \cite{Zflows} for more details. We see that $\calD$ is the graded domain $\calD = U^{(\tilde{m}_j)}$ where $\tilde{m}_j = m_j + \delta_{(-k)j}$ with coordinates $\{x^i\}_{i = 1}^m$ from $U^{(m_j)}$ and $\tau$, $|\tau| = -k$, from $\R[-k]$.

    Using the same arguments as in the previous case, we may again assume (by relabeling the coordinates $x^i$ for $i \in \{r, r+1, \dotsc, m\}$ if necessary) that $|\tau| = |x^r|$, and that the collection of tangent vectors
    \begin{equation}\label{eq_basis_nonzerodeg}
        \rest{\partial_{x^1}}{\bld{0}}, \dotsc, \rest{\partial_{x^{r-1}}}{\bld{0}}, \rest{X}{\bld{0}}, (T_{(\bld{0},0)}\theta_X)(\rest{\partial_{x^{r+1}}}{(\bld{0},0)}), \dotsc, (T_{(\bld{0},0)}\theta_X)(\rest{\partial_{x^m}}{(\bld{0},0)}),
    \end{equation}
    forms a basis for $T_{\bld{0}}U^{(m_j)}$. We may define a graded smooth map $\phi : U^{(m_j)} \to U^{(\tilde{m}_j)}$ as $\uline{\phi} = \id_U$ and 
    \begin{equation}
        \phi^\ast(\tau) = x^r, \qquad
        \phi^\ast(x^r) = 0, \qquad
        \phi^\ast(x^i) = x^i, \quad \text{for }i \neq r.
    \end{equation}
    One then proceeds to show that the composite map $\vartheta := \theta_X \circ \phi$ is a local diffeomorphism at $\bld{0} \in U$ in exactly the same manner as for the previous case. Note that now $\uline{\vartheta} = \id_U$. Thus, there exist some $V^\prime, U^\prime \subseteq U$ such that $\rest{\vartheta}{V^\prime} : (V^\prime)^{(m_j)} \to (U^{\prime})^{(m_j)}$ is a diffeomorphism, whose inverse is the desired coordinate map $\chi$. One final comment: the requirement $[X,X] = 0$ is nontrivial only for $|X|$ odd, in which case it is necessary for the existence of the flow $\theta_X$.
\end{proof}

Armed with Lemma \ref{lemma_local_frobenius} we may take on the local Frobenius theorem.

\begin{theorem}[Local Frobenius]\label{thm_local_frob}
Let $\scrD$ be an involutive distribution on $\calM$. Then around every point $p \in M$ there exist flat coordinates for $\scrD$.
\end{theorem}

\begin{proof}
    Let $\scrD$ be an involutive distribution on $\calM$ and consider $p \in M$. As this is a local statement, let us assume $\calM = U^{(m_j)}$ and let $p = \bld{0}\in U$. We shall prove the theorem via induction on the total rank $r := \sum_{j \in \Z}r_j$ of the distribution $\scrD$.

    \textbf{The base case}. Let $r = 1$. Thus every point of $U$ has a neighborhood on which $\scrD$ has a local frame consisting of a single vector field. By restricting $U$ around $\bld{0}$ if necessary, let $X$ form a frame for $\scrD$. By Proposition \ref{prop_commuting_frame} we may assume $[X,X] = 0$. From Lemma \ref{lemma_local_frobenius}, for the case $r = 1$, then immediately follows the existence of flat coordinates for $\scrD$.

    \textbf{The induction step}. Let the total rank of $\scrD$ be $r > 1$ and let the Local Frobenius theorem hold for all involutive distributions of total rank $r - 1$. Let $X_1, \dotsc, X_r$ be some commuting local frame for $\scrD$ around $\bld{0}$ whose existence is ensured by Proposition \ref{prop_commuting_frame}. Restrict $U$ around $\bld{0}$ so that it becomes a global frame. Since $X_1, \dotsc, X_{r - 1}$ are linearly independent at every point of $U$, we may define a distribution $\scrD^\prime$ as their $\cifty_{(m_j)}$-span. 
    
    The distribution $\scrD^\prime$ is involutive, hence let us use the induction hypothesis (and once more restrict $U$ around $\bld{0}$ if necessary) to find coordinates on $U^{(m_j)}$ which are flat for $\scrD^\prime$. Let $\phi : U^{(m_j)} \to V^{(m_j)}$ be the corresponding coordinate map. Let us strengthen the inductive hypothesis to assume that $\bld{0} \in V$, $\uline{\phi}(\bld{0}) = \bld{0}$ and that\begin{equation}\label{eq_local_frob_related_VFs}
        X_i \sim_{\phi} \partial_{x^i},
    \end{equation}
    for every $i \in \{1, \dotsc, r-1\}$, where $\{x^i\}_{i = 1}^m$ are the standard coordinates on $V^{(m_j)}$. Indeed, for $r = 1$ case this was satisfied due to Lemma \ref{lemma_local_frobenius} and we shall show that this property carries through the induction step.

    Consider that the $r$ vector fields $\partial_{x^1}, \dotsc, \partial_{x^{r-1}}, \phi_\ast X_r$ on $V^{(m_j)}$, where $\phi_\ast X_r := (\phi^{-1})^\ast \circ X_r \circ \phi^\ast$, are linearly independent at every point of $V$ and recall that $[X_i, X_j] = 0$ for all $i,j \in \{1, \dotsc, r\}$. Then from (\ref{eq_local_frob_related_VFs}), from $X_r \sim_\phi \phi_\ast X_r$ and from the fact that $\phi$ is a diffeomorphism it follows that 
    \begin{equation}
       [\phi_\ast X_r, \phi_\ast X_r] = 0, \qquad \text{and} \qquad   [\partial_{x^i}, \phi_\ast X_r] = 0, 
    \end{equation}
    for all $i\in\{1, \dotsc, r-1\}$. It so happens that these are precisely the requirements of Lemma \ref{lemma_local_frobenius}, which gives a subset $V^\prime \in \op_{\bld{0}}(V)$ and a diffeomorphism $\chi : (V^\prime)^{(m_j)} \to  W^{(m_j)}$ for some $W \in \op_{\bld{0}}(\R^{m_0})$, such that 
    \begin{equation}
        \rest{(\phi_\ast X_r)}{V^\prime} \sim_{\chi} \partial_{y^r}, \qquad \text{and} \qquad \rest{\partial_{x^i}}{V^\prime} \sim_{\chi} \partial_{y^i},
    \end{equation}
    for all i $\in \{1, \dotsc, r-1\}$, where $\{y^i\}$ are the standard coordinates on $W^{(m_j)}$. Denote $U^\prime := \uline{\phi}^{-1}(V^\prime)$ and note that $\bld{0} \in U^\prime$. Then the coordinates introduced by the composite diffeomorphism
    \begin{equation}
        \chi \circ \rest{\phi}{U^\prime} : (U^\prime)^{(m_j)} \to W^{(m_j)},
    \end{equation}
    are precisely the flat coordinates for $\scrD$. We end the induction step by observing that $\uline{\chi} (\uline{\phi}(\bld{0})) = \bld{0}$, and that $X_i \sim_{\chi \circ \rest{\phi}{U^\prime}} \partial_{y^i}$ for all $i \in \{1, \dotsc, r\}$.
\end{proof}

\begin{remark}
    In smooth manifold theory, every rank 1 distribution is involutive. In $\Z$-graded geometry there is $[X,X] = 2X^2$ for an odd vector field $X$. A distribution with rank $(\delta_{j\ell})_{j \in \Z}$ may therefore be non-involutive when $\ell$ is odd.  A non-involutive total rank 1 distribution is shown in Example \ref{example_integrable_noninvolutive_distr}.
\end{remark}

\section{Tangent Vector Fields and Integral Submanifolds}\label{section_3}

Recall that a graded smooth map $\iota : \calR \to \calM$ is an \textbf{immersion} if $T_p\iota$ is injective for every $p \in R$. We say $\iota$ is an \textbf{embedding} if it is an immersion and $\uline{\iota}$ is an embedding. If $\uline{\iota}$ is injective and $\iota$ is an immersion or an embedding, we say that $(\calR, \iota)$ is an \textbf{immersed} or \textbf{embedded submanifold} of $\calM$.

\begin{lemma}\label{lemma_restrictively_related_VFs}
    Let $\phi : \calN \to \calM$ be a graded smooth map, $Y \in \scrX_\calN(N)$ and $X \in \scrX_\calM(M)$. Then the following are equivalent:
    \begin{enumerate}
        \item $Y \sim_\phi X$.
        \item There exists an open cover $\{V_\alpha\}_{\alpha \in I}$ of $N$ such that $\rest{Y}{V_\alpha} \sim_{\rest{\phi}{V_\alpha}} X$ for every $\alpha\in I$.
        \item $\rest{Y}{V} \sim_{\rest{\phi}{V}}\rest{X}{U}$ for every $V \in \op(N)$ and $U \in \op(M)$ such that $\uline{\phi}(V)\subseteq U$.
    \end{enumerate}
\end{lemma}
\begin{proof}
    The implications $\mathit{3.  \implies 1.}$ and $\mathit{1. \implies 2.}$ are trivial. Let us show $\mathit{1. \implies 3.}$ Take some $U \in \op(M)$ and $V \in \op(N)$ such that $\uline{\phi}(V) \subseteq U$ and consider $\rest{\phi}{V} : \rest{\calN}{V} \to \rest{\calM}{U}$. We need to show that for every $f \in \cifty_\calM(U)$ there is
    \begin{equation}
        \rest{Y}{V} \left((\rest{\phi}{V})^{\ast}(f) \right) = (\rest{\phi}{V})^\ast ( \rest{X}{U}f).
    \end{equation}
     Fix some arbitrary $p\in V$ and let $\lambda \in \cifty_\calM(M)$ be a smooth bump function on $M$ supported in $U$ such that $\rest{\lambda}{U^\prime} = 1$ for some $U^\prime \in \op_{\uline{\phi}(p)}(U)$. Let $\lambda \cdot f \in \cifty_\calM(M)$ denote the graded smooth function defined by $\parest{\lambda \cdot f}{U} = \rest{\lambda}{U} \, f$ and $\parest{\lambda \cdot f}{M \setminus \supp(\lambda)} = 0$. A direct calculation yields
    \begin{equation}
    \begin{split}
        \left[\rest{Y}{V} \left(\rest{\phi}{V}^{\ast}(f) \right)\right]_p 
        =\left[\rest{Y}{V \cap \uline{\phi}^{-1}(U')}(\rest{\phi}{V}^\ast(\rest{f}{U'})) \right]_p
        =\left[\rest{Y}{V \cap \uline{\phi}^{-1}(U')}(\rest{\phi}{V}^\ast(\parest{\lambda \cdot f}{U'})) \right]_p \\
        = \left[\rest{Y}{\uline{\phi}^{-1}(U')}(\phi^\ast(\parest{\lambda \cdot f}{U'})) \right]_p
        = \left[Y(\phi^\ast(\lambda \cdot f))\right]_p 
        = \left[\phi^\ast(X(\lambda \cdot f))\right]_p 
        = \left[(\rest{\phi}{V})^\ast(\rest{X}{U}f) \right]_p
        \end{split}
    \end{equation}
    As $p \in U$ was arbitrary, this shows the implication. Finally, $\mathit{2. \implies 1.}$ follows immediately from the fact that for any $f \in \cifty_\calM(M)$ there is
    \begin{equation}
        \parest{\phi^\ast(X(f))}{V_\alpha} = \rest{\phi}{V_\alpha}^\ast(X(f)) = \rest{Y}{V_{\alpha}}(\rest{\phi}{V_\alpha}^\ast(f)) = \parest{Y(\phi^\ast(f))}{V_\alpha},
    \end{equation}
    for any $\alpha \in I$.
\end{proof}

\begin{prop}\label{prop_injective_immersions}
    Let $(\calR, \iota)$ be an immersed submanifold of $\calM$. Then
    \begin{enumerate}
        \item The graded smooth map $\iota : \calR \to \calM$ is a monomorphism in the category of $\Z$-manifolds.
        \item For any $X \in \scrX_\calM(U)$ there can be at most one $Y \in \scrX_\calR(\uline{\iota}^{-1}(U))$ such that $Y \sim_\iota X$.
    \end{enumerate}
\end{prop}
\begin{proof}
    Ad $\mathit{1.}$ Consider two graded smooth maps $\phi, \psi : \calN \to \calR$ such that $\iota \circ \phi = \iota \circ \psi$. We need to show that $\phi = \psi$. The injectivity of $\uline{\iota}$ implies that $\uline{\phi} = \uline{\psi}$, so we need only show that for any $V \in \op(R)$ and $f \in \cifty_\calR(V)$ there is $\phi^\ast(f) = \psi^\ast(f)$. This is equivalent to demonstrating the equality of the stalks $[\phi^\ast(f)]_p = [\psi^\ast(f)]_p$ for an arbitrary $p \in \uline{\phi}^{-1}(V) \equiv \uline{\psi}^{-1}(V)$.

    The key is \cite[Proposition 7.5]{GTGM}, which is the $\Z$-graded analog of the extension lemma for smooth functions on an embedded submanifold, e.g. \cite[Lemma 5.34]{ItSM}. According to \cite{GTGM}, there is $V^\prime \in \op_{\uline{\phi}(p)}(V)$ and some $O \in \op(M)$ such that $\uline{\iota}(V^\prime) \subseteq O$ with the property that $\rest{\iota}{V^\prime}^\ast : \cifty_\calM(O) \to \cifty_\calR(V^\prime)$ is surjective. This follows from the fact that every immersion is locally an embedding. Hence there exists some $g \in \cifty_\calM(O)$ such that
    \begin{equation}
        \rest{\iota^\ast(g)}{V^\prime} = \rest{f}{V^\prime},
    \end{equation}
    where we recall that $\rest{\iota}{V^\prime}^\ast(g) = \rest{\iota^\ast(g)}{V^\prime}$ by definition of $\rest{\iota}{V^\prime}$. Consequently,
    \begin{equation}
        \left[ \phi^\ast(f) \right]_p = \left[ \phi^\ast(\iota^\ast(g)) \right]_p = \left[ \psi^\ast(\iota^\ast(g))\right]_p = \left[ \psi^\ast(f) \right]_p,
    \end{equation}
    where we used the assumption $\iota \circ \phi = \iota \circ \psi$.
    
    Ad $\mathit{2.}$ Let $Y,Z \in \scrX_\calR(\uline{\iota}^{-1}(U))$ be two vector fields $\iota$-related to $X$. Denote $V:= \uline{\iota}^{-1}(U)$ and let $q \in V$ be arbitrary. As before, there exists some $V^\prime \in \op_q(V)$ and $O \in \op(U)$ such that $\uline{\iota}(V^\prime) \subseteq O$ and $\rest{\iota}{V^\prime}^\ast : \cifty_\calM(O) \to \cifty_\calR(V^\prime)$ is surjective. Hence for any $f \in \cifty_\calR(V)$ there exists $g \in \cifty_\calM(O)$ such that $\rest{f}{V^\prime} = \rest{\iota}{V^\prime}^\ast(g)$ and we can write
    \begin{equation}
    \begin{split}
        \left[Y(f) \right]_q &=  \left[\rest{Y}{V^\prime} (\rest{f}{V^\prime}) \right]_q =  \left[\rest{Y}{V^\prime} (\rest{\iota}{V^\prime}^\ast(g)) \right]_q = \left[\rest{\iota}{V^\prime}^\ast(\rest{X}{O}\, g) \right]_q \\ 
        &= \left[\rest{Z}{V^\prime} (\rest{\iota}{V^\prime}^\ast(g)) \right]_q = \left[Z(f)\right]_q,
    \end{split} 
    \end{equation}
    where we used Lemma \ref{lemma_restrictively_related_VFs}. As $q \in V$ was arbitrary, we have $Y = Z$.
\end{proof}

\begin{definition}\label{def_tangent_vector_field}
Let $(\calR, \iota)$ be an immersed submanifold of $\calM$ and $U \in \op(M)$. We say that a vector field $X \in \scrX_\calM(U)$ is \textbf{tangent to $\calR$} if there exists $Y \in \scrX_\calR(\uline{\iota}^{-1}(U))$ such that $Y \sim_{\iota} X$. In other words, if for all $f \in \cifty_\calM(U)$ there holds $Y\left(\iota^\ast(f)\right) = \iota^\ast\left(X(f)\right)$.
\end{definition}

\begin{remark}\label{remark_tangent_emptyset}
    What happens if $U \cap \uline{\iota}(R) = \emptyset$? Since $\scrX_\calR(\emptyset) = 0$ and $\cifty_\calR(\emptyset) = 0$, every $X \in \scrX_\calM(U)$ is tangent to $\calR$ by definition.
\end{remark}

\begin{lemma}\label{lemma_locally_tangent}
    Let $(\calR, \iota)$ be an immersed submanifold of $\calM$ and consider $X \in \scrX_\calM(U)$ for some $U \in \op(M)$. Then the following are equivalent:
    \begin{enumerate}
        \item $X$ is tangent to $\calR$.
        \item There exists an open cover $\{U_\alpha\}_{\alpha \in I}$ of $U$ such that $\rest{X}{U_\alpha}$ is tangent to $\calR$ for every $\alpha \in I$. 
        \item $\rest{X}{U^\prime}$ is tangent to $\calR$ for every $U^\prime \in \op(U)$.
    \end{enumerate}
\end{lemma}
\begin{proof}
    The implication $\mathit{1.} \implies \mathit{3.}$ follows from Lemma \ref{lemma_restrictively_related_VFs} whenever $U^\prime \cap \uline{\iota}(R) \neq \emptyset$. If $U^\prime \cap \uline{\iota}(R) = \emptyset$, see Remark \ref{remark_tangent_emptyset}. The implication $\mathit{3.} \implies \mathit{2}.$ is trivial. Let us show  $\mathit{2.} \implies \mathit{1.}$. By assumption, there exists $Y_\alpha \in \scrX_\calR(V_\alpha)$, where $V_\alpha := \uline{\iota}^{-1}(U_\alpha)$, such that $Y_\alpha \sim_{\iota} \rest{X}{U_\alpha}$. For any two $\alpha,\beta \in I$ denote $V_{\alpha \beta} := V_\alpha \cap V_\beta$ and $U_{\alpha \beta}$ similarly. Then by Lemma \ref{lemma_restrictively_related_VFs} both $\rest{Y_\alpha}{V_{\alpha \beta}}$ and $\rest{Y_\beta}{V_{\alpha \beta}}$ are $\iota$-related to $\rest{X}{U_{\alpha \beta}}$, hence by Proposition \ref{prop_injective_immersions} $\rest{Y_\alpha}{V_{\alpha \beta}} = \rest{Y_\beta}{V_{\alpha \beta}}$. Consequently, there exists $Y \in \scrX_\calR(V)$ such that $\rest{Y}{V_\alpha} = Y_\alpha$ for every $\alpha \in I$. The reader is left to verify that $Y \sim_\iota X$.
\end{proof}

\begin{definition}\label{def_integral_submanifold_1} \label{definition_integrable_distribution}
Let $\scrD$ be a rank $(r_{-j})$ distribution on $\calM$ and let $(\calR, \iota)$ be an immersed submanifold of $\calM$. We say that $(\calR, \iota)$ is a \textbf{fiberwise integral submanifold} for $\scrD$ if for every point $p \in R$ there is $\scrD_{\uline{\iota}(p)} = (T_p\iota)(T_p\calR)$, where $\scrD_{\uline{\iota}(p)}$ denotes the fiber of $\scrD$ at $\uline{\iota}(p)$. We say that $(\calR, \iota)$ is an \textbf{integral submanifold} for $\scrD$ if the dimension of $\calR$ is $(r_j)$ and every $X \in \scrD(U)$ is tangent to $\calR$ for every $U \in \op(M)$. We say that a distribution $\scrD$ is \textbf{integrable}, if for every point $p \in M$ there exists an integral submanifold $(\calR, \iota)$ for $\scrD$ such that $p \in \uline{\iota}(R)$.
\end{definition}

In Proposition \ref{prop_equivalently_integral}, we will see an equivalent definition using the pullback of the distribution. The reader can easily verify that every integral submanifold is fiberwise integral; however, unlike in ordinary geometry, not every fiberwise integral submanifold is integral. This is shown in Example \ref{example_strongly_and_weakly_integral_submanifolds}. 

\begin{lemma}
    \label{lemma_equivalent_strongly_integral}
Let $\scrD$ be a distribution on $\calM$ and $(\calR, \iota)$ an immersed submanifold of $\calM$. Then the following are equivalent:
\begin{enumerate}
\item For every $U \in \op(M)$, every $X \in \scrD(U)$ is tangent to $\calR$.
\item Every $X \in \scrD(M)$ is tangent to $\calR$.
\item For every point $p \in \uline{\iota}(R)$ there exists $U \in \op_p(M)$ such that every $X \in \scrD(U)$ is tangent to $\calR$.
\item For every point $p \in \uline{\iota}(R)$ there exists $U \in \op_p(M)$ together with a frame $X_1, \dotsc, X_r$ for $\scrD(U)$ such that each $X_i$ is tangent to $\calR$.
\end{enumerate}
\end{lemma}
\begin{proof}
    The implications $\mathit{1.} \implies \mathit{2.} \implies \mathit{3.}$ are trivial. To show the implication $\mathit{3.} \implies \mathit{4.}$ consider some $p \in M$ and a frame $X_1, \dotsc, X_r$ for $\scrD$ on some $U \in \op_p(M)$. From point $\mathit{3.}$ take some $V \in \op_p(M)$ such that every $X \in \scrD(V)$ is tangent to $\calR$. By Lemma \ref{lemma_locally_tangent}, we may assume $V \subseteq U$. Thus, simply restrict the frame vector fields $X_1, \dotsc, X_r$ to $V$. The implication $\mathit{4.} \implies \mathit{3.}$ follows immediately from the fact that if $Y \sim_\iota X$ and $\tilde{Y} \sim_\iota \tilde{X}$, then
    \begin{equation}\label{eq_cifty_comb_related_VFs}
     Y + \iota^\ast(f) \,\tilde{Y}   \sim_\iota X + f\tilde{X}.
    \end{equation}
    Finally, let us show the implication $\mathit{3.} \implies \mathit{1.}$ Consider some $U \in \op(M)$ and $X \in \scrD(U)$. Without loss of generality we may assume that $\uline{\iota}(R) \cap U \neq \emptyset$, see Remark \ref{remark_tangent_emptyset}. Fix some arbitrary $p \in \uline{\iota}(R) \cap U$. Extend $X$ to $\tilde{X}\in \scrX_\calM(M)$ using a smooth bump function $\lambda \in \cifty_\calM(M)$ which is supported in $U$ and $\rest{\lambda}{W} = 1$ on some $W \in \op_p(U)$. Explicitly $\tilde{X}|_{U} = \rest{\lambda}{U} \,X$ and $\tilde{X}|_{M \setminus \supp(\lambda)} = 0$. Next, by assumption there exists some $V \in \op_p(M)$ such that every vector field in $\scrD(V)$ is tangent to $\calR$. In particular, $\tilde{X}|_{V}$ is tangent to $\calR$ and by Lemma \ref{lemma_locally_tangent} so is $\tilde{X}|_{V \cap W} = X|_{V \cap W}$. As $p$ was arbitrary, we know that for every $p \in \uline{\iota}(R) \cap U$ there exists some $U_p \in \op_p(U)$ such that $X|_{U_p}$ is tangent to $\calR$. Denote $U^\prime := \cup_{p}U_p$. From Lemma \ref{lemma_locally_tangent} it follows that $\rest{X}{U^\prime}$ is tangent to $\calR$.

    Note that $U^\prime \supseteq \uline{\iota}(R)\cap U$, hence $\uline{\iota}^{-1}(U^\prime) = \uline{\iota}^{-1}(U)$. Consequently, for any $g \in \cifty_\calM(U)$ there is $\iota^\ast(g) = \iota^\ast(\rest{g}{U^\prime})$. Since we know there exists $Y \in \scrX_\calR(\uline{\iota}^{-1}(U))$ such that $Y \sim_\iota \rest{X}{U^\prime}$, we can write
    \begin{equation}
        Y (\iota^\ast(f)) = Y(\iota^\ast(\rest{f}{U^\prime})) = \iota^\ast (\rest{X}{U^\prime} (\rest{f}{U^\prime})) = \iota^\ast(\parest{Xf}{U^\prime}) = \iota^\ast(Xf),
    \end{equation}
    for any $f\in\cifty_\calM(U)$. In other words, $Y \sim_\iota X$.
\end{proof}

\begin{lemma}\label{lemma_locally_strongly_integral}
Let $\scrD$ be a distribution on $\calM$, let $(\calR, \iota)$ be an immersed submanifold of $\calM$ and let there exist an open cover $\{R_\alpha\}_{\alpha \in I}$ of $R$ such that every $(\calR|_{R_\alpha}, \iota|_{R_\alpha})$ is an integral submanifold for $\scrD$. Then $(\calR, \iota)$ is an integral submanifold for $\scrD$.
\end{lemma}
\begin{proof}
    Fix a vector field $X \in \scrD(M)$. By Lemma \ref{lemma_equivalent_strongly_integral} it is sufficient to show that $X$ is tangent to $\calR$. By assumption, for every $\alpha \in I$ there exists $Y_\alpha \in \scrX_{\calR}(R_\alpha)$  such that $Y_\alpha \sim_{\rest{\iota}{R_\alpha}} X$. By Lemma \ref{lemma_restrictively_related_VFs} this implies that $\rest{Y_\alpha}{R_{\alpha\beta}}\sim_{\rest{\iota}{R_{\alpha\beta}}} X$ for any $\alpha, \beta \in I$, where $R_{\alpha\beta} = R_\alpha \cap R_\beta$. It follows from Proposition \ref{prop_injective_immersions} that $\rest{Y_{\alpha}}{R_{\alpha\beta}} = \rest{Y_\beta}{R_{\alpha\beta}}$, hence the collection $\{Y_\alpha\}_{\alpha \in I}$ glues together a vector field $Y \in \scrX_\calR(R)$ such that $\rest{Y}{R_\alpha} = Y_\alpha \sim_{\rest{\phi}{R_\alpha}}X$. Using Lemma \ref{lemma_restrictively_related_VFs} one more time, one obtains that $Y \sim_{\iota} X$.
\end{proof}

In our examples we will use the following proposition from \cite{GTGM}, which gives a simple criterion of whether a vector field is tangent to a given embedded submanifold.

\begin{prop}[{\cite[Proposition 7.10.]{GTGM}}]\label{prop_embedded_tangent_criterion}
Let $(\calR, \iota)$ be an embedded submanifold of $\calM$, $U \in \op(M)$ and $X \in \scrX_\calM(U)$. Then $X$ is tangent to $\calR$ if and only if for every $f \in \cifty_\calM(U)$ there holds the implication
\begin{equation}\label{eq_impl1}
\iota^\ast(f) = 0 \implies \iota^\ast \left( Xf \right) = 0.
\end{equation}
In other words, if and only if $X$ preserves the kernel of $\iota^\ast$.
\end{prop}

There is another, perhaps more natural viewpoint on integral and fiberwise integral submanifolds obtained by comparing the tangent bundle of the submanifold to the pullback of the distribution. Let us recall the relevant theory of pullback vector bundles. First, let us set the notation with a diagram.
    \begin{equation}\label{diag_pullback_tangent_bundle}
        \begin{tikzcd}
            {T\calR}
            \arrow[d, "{}"]
            \arrow[rrd, "{T\varphi}", bend left]
            \arrow[rd, "{\hat{T}\varphi}", dashed]
                &{}
                    &{} \\
            {\calR}
            \arrow[rd, bend right, "{\id}"]
                &{\varphi^{\star} T\calM}
                \arrow[r, "{\hat{\varphi}}"]
                \arrow[d, "{}"]
                    &{T \calM}
                    \arrow[d]\\
            {}
                &{\calR}
                \arrow[r, "{\varphi}"]
                    &{\calM}
        \end{tikzcd}.
    \end{equation}
    The graded vector bundle (GVB) morphism $\hat{T}\varphi$ over the identity $\calR \to \calR$ is induced by the universal property of the pullback bundle $\varphi^{\star} T\calM$, see e.g. \cite[Example 5.5]{3GVB}. The fact that $\hat{T}\varphi$ is a GVB morphism over the identity lets us view it as a morphism of sheaves of $\cifty_\calR$-modules $\hat{T}\varphi : \scrX_\calR \to \varphi^{\star} \scrX_\calM$, hence it makes sense to write $(\hat{T}\varphi)(Y)$ for any $Y \in \scrX_\calR(U)$ and any $U \in \op(R)$.
    Note that we use the subtly different $\varphi^\star T\calM$ from the more common $\varphi^\ast T \calM$ to avoid confusion with the pullback by $\varphi$, i.e. the sheaf morphism $\varphi^\ast : \cifty_\calM \to \uline{\varphi}_\ast \cifty_\calR$.

    Next, let us recall how one obtains the pullback section $\varphi^\star X \in (\varphi^\star \scrX_\calM)(\uline{\varphi}^{-1}(V)) \equiv (\varphi_\ast \varphi^\star\scrX_\calM)(V)$ for any $X \in \scrX_\calM(V)$. Here, for any sheaf of $\cifty_\calR$-modules $\scrF$ we denote as $\varphi_\ast \scrF$ the sheaf of $\cifty_\calM$-modules given by $(\varphi_\ast \scrF)(V) := (\uline{\varphi}_\ast\scrF)(V) \equiv \scrF(\uline{\varphi}^{-1}(V))$ and $f \cdot s := \varphi^\ast(f)\, s$ for any $V \in \op(M)$, $s \in (\varphi_\ast \scrF)(V)$ and $f \in \cifty_\calM(V)$. It is a very useful fact that $\varphi_\ast$ and $\varphi^\star$ can be seen as functors
    \begin{equation}
    \begin{tikzcd}
        \Mod_{\cifty_\calR}
        \arrow[r, yshift = 3pt, "{\varphi_\ast}"]
        \arrow[r, leftarrow, yshift = -3pt, "{\varphi^\star}"']
        & \Mod_{\cifty_\calM}
    \end{tikzcd},
    \end{equation}
    between the categories of sheaves of $\cifty_\calR$- and $\cifty_\calM$- modules. What is more, $\varphi^\star$ is a left-adjoint of $\varphi_\ast$. In other words, there is a natural isomorphism
    \begin{equation}
        \Mod_{\cifty_\calR}(\varphi^\star -, -) \cong \Mod_{\cifty_\calM}(-, \varphi_\ast -),
    \end{equation}
    see e.g. \cite[Chapter II, Section 5]{Hartshorne1977}. Consider the unit of the adjunction $\eta$, which is a natural transformation $\eta : \id \to \varphi_\ast \varphi^\star $. For the concrete object $\scrX_\calM \in \Mod_{\cifty_\calM}$, it gives a morphism of sheaves of $\cifty_\calM$-modules
    \begin{equation}
        \eta_{\scrX_\calM} : \scrX_\calM \to \varphi_\ast \varphi^\star \scrX_\calM.
    \end{equation}
    For any $V \in \op(M)$  and $X \in \scrX_\calM(V)$, we obtain the pullback section as
    \begin{equation}
        \varphi^\star X \equiv X^\star := \eta_{\scrX_\calM}(X).
    \end{equation}
    Also note  that there is a canonical isomorphism $\varphi^\star(T^\ast \calM)\cong (\varphi^\star T \calM)^\ast$.  In fact, let $\alpha^i$ be a dual section to $\partial_{x^i}$, i.e. $\alpha^{i}(\partial_{x^j}) = {\delta^i}_j$. In other words $\alpha^i = (-1)^{|x^i|}\, \mathrm{d}x^i$ (no sum). Then $\varphi^{\star}(\alpha^i)$ is the dual section to $\varphi^\star(\partial_{x^i})$, see \cite[(607)]{GTGM}. Furthermore, the induced GVB morphism $\hat{\varphi}^\ast : \scrX_\calM^\ast \to \varphi_\ast ((\varphi^\star\scrX_\calM)^\ast)$ is also the unit of the adjunction, just valued at the object $\Omega^1_\calM \cong \scrX_\calM^\ast$, i.e.
    \begin{equation}
        \hat{\varphi}^\ast \cong \eta_{\Omega^1_\calM} : \Omega^1_\calM \to \varphi_\ast \varphi^\star \Omega^1_\calM.
    \end{equation}
Consequently, $\varphi^{\star}(\alpha^i) = \hat{\varphi}^\ast(\alpha^i)$. The next lemma tells us how to express $(\hat{T}\varphi)(Y)$ in the pullback frame.

\begin{lemma}\label{lemma_locally_pullback_section}
Let $\calR$ and $\calM$ be $\Z$-manifolds, $\varphi : \calR \to \calM$ a graded smooth map and consider some $Y \in \scrX_\calR(R)$. Let $\{y^a\}_{a = 1}^r$ be coordinates on some $U \in \op(R)$. Furthermore, let $\{x^i\}_{i = 1}^m$ be coordinates on some $V \in \op(M)$ such that $\uline{\varphi}^{-1}(V) \cap U \neq \emptyset$. Denote $U' := \uline{\varphi}^{-1}(V) \cap U$ and $Y|_{U'} = Y^a \, \partial_{y^a}$. Then 
\begin{equation}
(\hat{T}\varphi)(Y|_{U'}) =  Y^a  \,  \dd{(\varphi|_{U'}^\ast x^i)}{y^a} \, \partial_{x^i}^\star|_{U'},
\end{equation}
where $\partial_{x^i}^\star$ denotes the pullback section of $\partial_{x^i}$.
\end{lemma}

\begin{proof}
It is enough to show that
    \begin{equation}\label{eq_eeee}
        (\hat{T}\varphi)(\partial_{y^a}|_{U'}) =  \dd{(\varphi|_{U'}^\ast x^i)}{y^a} \, \partial_{x^i}^\star|_{U'} ,
    \end{equation}
    for every $a \in \{1, \dotsc, r\}$, or equivalently that 
    \begin{equation}\label{eq_makarena}
    \rest{\left[(\hat{T}\varphi)^\rmT (\varphi^\star(\alpha^i))\right]}{U'}(\partial_{y^a}|_{U'}) = (-1)^{|x^i| \, (|x^i| - |y^a|)} \, \dd{(\varphi|_{U'}^\ast x^i)}{y^a},
    \end{equation}
 	where $(\hat{T}\varphi)^\rmT : \varphi^\star \Omega^1_{\calM} \to \Omega^1_{\calR}$ is the transpose of $\hat{T}\varphi : \scrX_\calR \to \varphi^\star \scrX_\calM$. But we know that 
\begin{equation}
\begin{split}
(\hat{T}\varphi)^\rmT (\varphi^\star(\alpha^i)) &= (\hat{T}\varphi)^\rmT (\hat{\varphi}^\ast(\alpha^i)) = (T\varphi)^\ast (\alpha^i) = (-1)^{|x^i|} (T\varphi)^\ast (\mathrm{d} x^i) \\
&= (-1)^{|x^i|}\, \mathrm{d}y^b \, \dd{(\varphi|_{U'}^\ast x^i)}{y^b} = (-1)^{|x^i| + |y^b| \, (|x^i| - |y^b|) + |y^b|} \,\dd{(\varphi|_{U'}^\ast x^i)}{y^b} \beta^b,
\end{split}
\end{equation}
where $\alpha^i$ and $\beta^b$ are dual sections to $\partial_{x^i}$ and $\partial_{y^b}$, respectively.
A little calculation tells us that
\begin{equation}
|x^i| + |y^b| \, (|x^i| - |y^b|) + |y^b| = |x^i|(|x^i| - |y^b|) \quad \mod 2,
\end{equation}
which shows (\ref{eq_makarena}), hence (\ref{eq_eeee}) and so finishes the proof.
\end{proof}

\begin{corollary}\label{cor_related_VFs} 
Consider $X \in \scrX_\calM(M)$,  $Y \in \scrX_\calR(R)$ and $\varphi : \calR \to \calM$. Then
there is $Y \sim_\varphi X$ if and only if $(\hat{T}\varphi)(Y) = \varphi^\star X$.
\end{corollary}
\begin{proof}
By definition, $Y \sim_\varphi X$ means that $Y \circ \varphi^\ast = \varphi^\ast \circ X$. Locally, in the notation of Lemma \ref{lemma_locally_pullback_section}, this translates to $Y|_{U'}(\varphi^\ast (x^i)) = \varphi^\ast(X^i|_{V})$. A different way of writing this is
\begin{equation}
Y^a|_{U'}  \,  \dd{(\varphi|_{U'}^\ast x^i)}{y^a} = \varphi^\ast(X^i|_{V}).
\end{equation}
From Lemma \ref{lemma_locally_pullback_section} we know that 
\begin{equation}
(\hat{T}\varphi)(Y|_{U'}) = Y^a|_{U'}  \,  \dd{(\varphi|_{U'}^\ast x^i)}{y^a} \, \partial_{x^i}^\star|_{U'},
\end{equation}
and from the fact that the assignment $X \mapsto \varphi^\star X$ is a morphism of sheaves of $\cifty_\calM$-modules we know that
\begin{equation}
(\varphi^\star X)|_{U'} = \varphi^\ast(X^i)|_{U'} \,\partial_{x^i}^\star|_{U'} = \varphi^\ast(X^i|_{V})\, \partial_{x^i}^\star|_{U'}.
\end{equation}
The result follows.
\end{proof}

\begin{prop}\label{prop_equivalently_integral}
    Let $(\calR, \iota)$ be an immersed submanifold of $\calM$ and let $\scrD$ be a distribution on $\calM$. Then both $\im(\hat{T}\iota)$ and $\iota^\star \scrD$ are subbundles of $\iota^\star T\calM$ and the following holds: \vspace{-5pt}
\begin{enumerate}
    \item The graded vector bundles $\im(\hat{T} \iota)$ and $\iota^\star\scrD$ have the same fibers if and only if $(\calR, \iota)$ is a fiberwise integral submanifold for $\scrD$.
    \item There is $\im(\hat{T} \iota) = \iota^\star(\scrD)$ if and only if $(\calR, \iota)$ is an integral submanifold for $\scrD$.
\end{enumerate}
\end{prop}
\begin{proof}
The situation  can be drawn as
\begin{equation}\label{diag_integral_pullback}
    \begin{tikzcd}[column sep = small]
        {}
            &{}
                &{T\calR}
                \arrow[ld, "{\hat{T}\iota}"']
                \arrow[dddl, "{}"]
                \arrow[rrd, "{T\iota}"]
                    &{}
                        &{}
                            &{}\\
        {}
            &{\iota^\star T\calM}
            \arrow[rrr, "{\hat{\iota}}", crossing over]
            \arrow[dd, "{}"]
                &{}
                    &{}
                        &{T \calM}
                        \arrow[dd, "{}"]
                            &{} \\
        {\iota^\star \scrD}
        \arrow[dr, "{}"]
        \arrow[ru, "{}", hook]
        \arrow[rrr, "{}", crossing over]
            &{}
                &{}
                    &{\scrD}
                    \arrow[ru, "{}", hook]
                    \arrow[rd, "{}"]
                        &{}
                            &{} \\
        {}
            &{\calR}
            \arrow[rrr, "{\iota}"]
                &{}
                    &{}
                        &{\calM}
                            &{}
    \end{tikzcd}.
\end{equation}
By assumption $T \iota$ is a fiberwise injective GVB morphism. As a result, so is the induced morphism $\hat{T}\iota$ and, consequently, $\im(\hat{T}\iota)$ is a subbundle of $\iota^\star T\calM$ by \cite[Proposition 3.14]{3GVB}. Similarly one argues that $\iota^\star \scrD$ is also a subbundle of $\iota^\star T\calM$.

    $(\calR, \iota)$ is a fiberwise integral submanifold for $\scrD$ if and only if $\im(T_p \iota) = \scrD_{\uline{\iota}(p)}$ for every $p \in R$. Since the fiber map $\hat{\iota}_{(p)}$ is a linear isomorphism for every $p$, the claim follows from commutativity of the diagram (\ref{diag_integral_pullback}).
    \item If $\im(\hat{T} \iota) = \iota^\star(\scrD)$, then for every vector field $X \in \scrD(M)$ there is some $Y \in \scrX_\calR(R)$ such that $(\hat{T}\iota) (Y) = \iota^\star X$, which is the same as $Y \sim_\iota X$ by Corollary \ref{cor_related_VFs}. 
    
    Conversely, let $(\calR, \iota)$ be an integral submanifold for $\scrD$. Consider a local frame $\{X_k\}_{k = 1}^r$ for $\scrD$ on some $V \in \op(M)$ such that $\uline{\iota}^{-1}(V) \neq \emptyset$. By assumption, for every $k \in \{1, \dotsc, r\}$ there exists $Y_k \in \scrX_\calR(\uline{\iota}^{-1}(V))$ such that $Y_k \sim_\iota X_k$, i.e. $(\hat{T}\iota)(Y_k) = \iota^\star X_k$. In particular, this implies that $\{Y_k|_{q}\}_{k = 1}^r$ are linearly independent vectors for every $q \in \uline{\iota}^{-1}(V)$ and hence $\{Y_{k}\}_{k = 1}^{r}$ forms a local frame for $\scrX_\calR$ on $\uline{\iota}^{-1}(V)$. Consequently, so $\{(\hat{T}\iota)(Y_k)\}_{k = 1}^r$ forms a local frame for $\im(\hat{T}\iota)$ on $\uline{\iota}^{-1}(V)$. Finally, since $\iota^\star X_k$ form a local frame for $\iota^\star \scrD$, we find that the $\cifty_\calR(\uline{\iota}^{-1}(V))$-modules $(\im(\hat{T}\iota))(\uline{\iota}^{-1}(V))$ and $(\iota^\star \scrD)(\uline{\iota}^{-1}(V))$ contain each others generators, and hence they must be equal.

    By the above argument we know there exists an open cover $\{U_\alpha\}_{\alpha \in I}$ of $R$ where $(\im(\hat{T}\iota))(U_\alpha) = \left(\iota^\star \scrD \right)(U_\alpha)$,
    for every $\alpha \in I$. Indeed, simply set $U_\alpha := \uline{\iota}^{-1}(V_\alpha)$ for an open cover $\{V_\alpha\}_{\alpha \in I}$ of $\uline{\iota}(R)$ by sets $V_\alpha$ as above. This immediately implies equality of the sheaves of modules $\im (\hat{T}\iota) = \iota^\star \scrD$, see e.g. \cite[Remark 4.29]{3GVB}.
\end{proof}

Using this language of pullback bundles gives us an alternative way to search for fiberwise integral and integral submanifolds. The following example shows that an involutive distribution may, at the same time, admit fiberwise integral submanifolds that are integral, and those that are not.

\begin{example} \label{example_strongly_and_weakly_integral_submanifolds}
Let $\calM := \{\ast\}^{(m_j)}$ be a one point manifold with coordinates $\xi^1, \xi^2, \xi^3$ of degree 1 and $\eta$ of degree 3, and let $\calR := \{\star\}^{(r_j)}$ be a one point manifold with coordinates $\theta^1, \theta^2, \theta^3$ of degree 1. Let $\scrD$ be a distribution on $\calM$ generated by the coordinate vector fields $\{\partial_{\xi^i}\}_{i = 1}^3$. We will find all immersions $\iota : \calR \to \calM$ which make $(\calR, \iota)$ into a fiberwise integral or integral submanifold for $\scrD$. A general graded smooth map $\iota: \calR \to \calM$ is given by
\begin{equation}
\iota^\ast(\xi^j) = {A^j}_k \, \theta^k, \quad \text{and} \quad \iota^\ast(\eta) = B\, \theta^1\theta^2\theta^3,
\end{equation}
for some numbers ${A^j}_k, B \in \R$. We know that $\iota^\star\scrD$ is generated by the pullback sections $\{\partial_{\xi^i}^\star\}_{i = 1}^3$. The image of $\hat{T}\iota$ is in turn generated by sections $\{(\hat{T}\iota)(\partial_{\theta^i})\}_{i = 1}^3$. From Lemma \ref{lemma_locally_pullback_section} we have
\begin{equation}
    (\hat{T}\iota)(\partial_{\theta^i}) = \dd{\iota^\ast(\xi^j)}{\theta^i} \, \partial_{\xi^j}^\star + \dd{\iota^\ast(\eta)}{\theta^i} \, \partial_{\eta}^\star = {A^j}_i \, \partial_{\xi^j}^\star + (-1)^{i + 1} B\left( \prod_{k \neq i}\theta^k \right) \partial_{\eta}^\star.
\end{equation}
From this we immediately see that  $(\calR,\iota)$ is a fiberwise integral submanifold for $\scrD$ iff the numbers ${A^i}_j$ form an invertible matrix. Furthermore, it becomes an integral submanifold iff $B = 0$.
\end{example}

\section{Global Frobenius Theorem}\label{section_5}

In the classical setting, Global Frobenius theorem gives, for any involutive distribution and any point, a unique (maximal, connected) integral submanifold passing through this point. In our case, however, the best we can hope for is to require uniqueness up to some equivalence. We will use the following:

\begin{definition}\label{def_equivalent_submanifolds}
Let $(\calR, \iota)$ and $(\calR', \iota')$ be two immersed submanifolds of $\calM$. 
We say they are \textbf{equivalently immersed}  if there exists a graded diffeomorphism $\vartheta : \calR' \to \calR$ such that
\begin{equation}\label{eq_coincident_submanifolds}
\begin{tikzcd}
{}
	&{\calM}
		&{} \\
\calR
\arrow[ru, "{\iota}"]
	&{}
		&{\calR'}
		\arrow[lu, "{\iota'}"']
		\arrow[ll, "{\vartheta}"']
\end{tikzcd}
\end{equation}
commutes.
\end{definition}

Clearly this is an equivalence relation. Next two lemmas ensure that it can be verified locally, and that, under this equivalence, the notion of being a fiberwise integral or integral submanifold behaves transitively.

\begin{lemma}\label{lma_similarly_immersed}
Let $(\calR, \iota)$ and $(\calR', \iota')$ be two immersed submanifolds of $\calM$. If there exists an open cover $\{V_\alpha\}_{\alpha \in I}$ of $R$ and $\{V'_\alpha\}_{\alpha \in I}$ of $R'$ such that $(\calR|_{V_\alpha}, \iota|_{V_\alpha})$ and $(\calR'|_{V'_\alpha}, \iota'|_{V'_\alpha})$ are equivalently immersed for every $\alpha \in I$, then $(\calR, \iota)$ and $(\calR', \iota')$ are equivalently immersed.
\end{lemma}
\begin{proof}
    By assumption there exists, for every $\alpha \in I$, a graded diffeomorphism $\vartheta_\alpha : \calR'|_{V'_\alpha} \to \calR|_{V_\alpha}$ such that $\iota|_{V_\alpha} \circ \vartheta_{\alpha} = \iota'|_{V'_\alpha}$. Note that for any $p \in V'_{\alpha \beta} := V'_{\alpha} \cap V'_\beta$, we have $\uline{\iota}(\uline{\vartheta_\alpha}(p)) = \uline{\iota}'(p) = \uline{\iota}(\uline{\vartheta_\beta}(p))$. Since $\uline{\iota}$ is injective, this implies $\uline{\vartheta_\alpha}|_{V'_{\alpha \beta}} = \uline{\vartheta_\beta}|_{V'_{\alpha \beta}}$. As a result we can write, for each $\alpha, \beta \in I$, that  
    \begin{equation}
        \rest{\iota}{V_{\alpha \beta}} \circ \rest{\vartheta_{\alpha}}{V'_{\alpha \beta}} = \rest{\iota'}{V'_{\alpha\beta}} = \rest{\iota}{V_{\alpha \beta}} \circ \rest{\vartheta_{\beta}}{V'_{\alpha \beta}}.
    \end{equation}
    By proposition \ref{prop_injective_immersions}, $\iota|_{V_{\alpha \beta}}$ is a monomorphism, thus the graded smooth maps $\vartheta_{\alpha}$ agree on overlaps, and so glue together the desired graded smooth map $\vartheta: \calR' \to \calR$. It is not difficult to see that $\uline{\vartheta}$ must be bijective. As $\vartheta$ is a local diffeomorphism, it is a diffeomorphism by \cite[Proposition 4.31]{GTGM}.
\end{proof}

\begin{lemma}\label{lemma_strongly_weakly_integral}
Let $(\calR, \iota)$ and $(\calR', \iota')$ be two equivalently immersed submanifolds of $\calM$ and let $\scrD$ be a distribution on $\calM$. If $(\calR, \iota)$ is a fiberwise integral or integral submanifold for $\scrD$, then the same is true for $(\calR', \iota')$.
\end{lemma}
\begin{proof}
    The verification is straightforward and we leave it to the reader.
\end{proof}

\begin{corollary}\label{cor_nonequivalently_integral}
For a given involutive distribution $\scrD$ on $\calM$ there may exist fiberwise integral submanifolds $(\calR, \iota)$, $(\calR', \iota')$ for $\scrD$, such that $\uline{\iota}(R) = \uline{\iota}'(R')$, which are not equivalently immersed.
\end{corollary}
\begin{proof}
    In Example \ref{example_strongly_and_weakly_integral_submanifolds} we encountered a one-point $\Z$-manifold $\calM$ with an involutive distribution that admits an integral submanifold (when $B = 0$), but also a fiberwise integral submanifold that is not integral (when $B \neq 0$). By Lemma \ref{lemma_strongly_weakly_integral}, two such submanifolds cannot be equivalently immersed.
\end{proof}

Using flat coordinates, whose existence is ensured by the Local Frobenius theorem, we find a local prototype of an integral submanifold for any involutive distribution.

\begin{example}\label{example_prototype_strongly_integral_submanifold}
Let $\scrD$ be an involutive distribution of rank $(r_{-j})$ on $\calM$ and fix some point $p_0 \in M$. From the Local Frobenius theorem we know that there exists a chart $(\phi, U)$ around $p_0$ introducing coordinates $(x^\mu, \xi^a)$ such that $\scrD(U)$ is spanned by $\{\partial_{x^\mu}\}_{\mu = 1}^{r_0}$ and $\{\partial_{\xi^a}\}_{a = 1}^{\hat{r}}$, i.e. flat coordinates for $\scrD$. Denote $\hat{U} := \uline{\phi}(U)$. We may assume that $\hat{U}\subseteq{\R^{m_0}}$ is a cube centered at $\uline{\phi}(p_0)  = \bld{0}$. 

Let $\pr_{r_0} : \R^{m_0} \to \R^{r_0}$ be the projection map on the first $r_0$ entries, denote $\hat{Q} = \pr_{r_0}(\hat{U}) \subseteq \R^{r_0}$ and consider the graded domain $\hat{Q}^{(r_{j})}$. Let us denote the coordinates on therein by $(y^\mu, \theta^a)$, in such a way that $|\theta^a| = |\xi^a|$ for all $a \leq \hat{r}$. 
Define a graded smooth map $\sigma : \hat{Q}^{(r_{j})} \to\hat{U}^{(m_j)}$ by $\uline{\sigma}(s_1, \dotsc, s_{r_0}) = (s_1, \dotsc, s_{r_0}, 0, \dotsc, 0)$, ${\sigma}^\ast(x^\mu) = x^\mu \circ \uline{\sigma}$, ${\sigma}^\ast(\xi^a) = \theta^a$ for $a \leq \hat{r}$ and ${\sigma}^\ast(\xi^a) = 0$ for $a > \hat{r}$.

Denote $\iota := \phi^{-1} \circ \sigma$. Clearly $\iota : \hat{Q}^{(r_{j})} \to \calM|_{U}$ is an immersion and $\scrX_{(r_{j})}$ and $\scrD$ have the same rank by construction. It is not hard to see that $\partial_{y^i} \sim_{\sigma} \partial_{x^i} \in \scrX_{(m_j)}(\hat{U})$, hence $\partial_{y^i} \sim_\iota \partial_{x^i} \in \scrX_\calM(U)$, for all $i \in \{1, \dotsc, r\}$. Consequently, $(\hat{Q}^{(r_{j})}, \iota)$ is an integral submanifold for $\scrD$.
\end{example}

\begin{corollary}\label{cor_involutive_integrable}
    Every involutive distribution on a $\Z$-manifold is integrable.
\end{corollary}
\begin{proof}
    Follows from the preceding example. Let us note that the converse implication does not hold, see Example \ref{example_integrable_noninvolutive_distr}.
\end{proof}

The rest of the section concerns itself with maximally extending these locally defined integral submanifolds. We have seen that an involutive distribution may admit two non-equivalently immersed fiberwise integral submanifolds. The next proposition shows that this is not the case for integral submanifolds. First a short lemma.

\begin{lemma}\label{lemma_frame_on_function_zero}
    Let $\calM$ be a $\Z$-manifold, $\{X_j\}_{j = 1}^m$ a global frame for $\scrX_\calM(M)$ and let $f \in \cifty_\calM(M)$ be a function such that $X_j(f) = 0$ for all $j \in \{1, \dotsc, m\}$. Then the following holds:
    \begin{enumerate}
        \item If $|f| = 0$, then $f$ is a locally constant function.
        \item If $|f| \neq 0$, then $f = 0$.
    \end{enumerate}
\end{lemma}
\begin{proof}
    Let us first elaborate what we mean by saying that $f \in \cifty_\calM(M)$ is locally constant. In $\Z$-graded geometry there is, in general, no canonical way to define a map $\calM \to M$. In other words, there is no canonical way to say that a function $f \in \cifty_\calM(M)$ is ``non-graded''. However, we can still say that a function $f \in \cifty_\calM(M)$ is \textbf{constant}, if it is a real multiple of the unit $1 \in \cifty_\calM(M)$. We say it is \textbf{locally constant} if for every point $p \in M$ there exists $U \in \op_p(M)$ such that $\rest{f}{U}$ is constant. Clearly a locally constant function is constant on every connected component of $M$.

    Now consider some $f$ such that $X_j(f) = 0$ for all $j$, and let $U \in \op(M)$ be some open set with coordinates $\{x^i\}_{i = 1}^m$.  Note that as $\{X_j|_{U}\}_{j = 1}^m$ form a frame for $\scrX_\calM(U)$, it follows that $\partial_{x^j}f|_{U} = 0$ for all $j \in \{1, \dotsc, m\}$. In particular this is true for all $j$ such that $|x^j| \neq 0$, hence in these coordinates $f|_{U}$ has no graded part. In other words, $f|_{U} = \uline{f}|_{U}$. Note that this equality makes sense only on a coordinate patch, see the paragraph above. From the fact that $\partial_{x^j}f_{U} = 0$ for all $j$ such that $|x^j| = 0$ it then follows that $\uline{f}|_{U}$ must be locally constant, and hence so must $f$. The case $|f| \neq 0$ follows the same reasoning. 
\end{proof}

\begin{prop}\label{prop_similarly_immersed_strongly_integral}
Let $(\calR, \iota)$ and $(\calR', \iota')$ be two integral submanifolds of an involutive distribution $\scrD$ on $\calM$ such that $\uline{\iota}(R) = \uline{\iota}'(R')$. Then $(\calR, \iota)$ and $(\calR', \iota')$ are equivalently immersed.
\end{prop}
\begin{proof}
    The idea of the proof is to locally relate both submanifolds to our template submanifold from Example \ref{example_prototype_strongly_integral_submanifold} and use this to construct a local diffeomorphism between them. Let us start by fixing a point $p_0 \in \uline{\iota}(R) = \uline{\iota'}(R')$. Let $U$ be some open neighborhood of $p_0 \in M$ such that there is a coordinate chart
    \begin{equation}
        \phi : \rest{\calM}{U} \to \hat{U}^{(m_j)},
    \end{equation}
    giving rise to flat coordinates for $\scrD$ on $U$, where $\hat{U}:= \uline{\phi}(U)$ is a cube and $\uline{\phi}(p_0) = \bld{0} \in \R^{m_0}$. The existence of this is ensured by the Local Frobenius theorem. Let us denote $V := \uline{\iota}^{-1}(U)$, $V' := \uline{\iota'}^{-1}(U)$ and carry over the rest of the notation from Example \ref{example_prototype_strongly_integral_submanifold}. The stage has been set:
    \begin{equation}
    \begin{tikzcd}
    {}&
	   {\rest{\calM}{U}}
	   \arrow[rr, "{\phi}"]&
	       {}&
	       {\hat{U}^{(m_j)}}
	       \arrow[d, "{\pi}"] \\
    {\rest{\calR}{V}}
    \arrow[ru, "{\rest{\iota}{V}}"]&
	   {}&
	   {\rest{\calR'}{V'}}
	   \arrow[lu, "{\rest{\iota'}{V'}}"']&
	       {\hat{Q}^{(r_{j})}}
    \end{tikzcd}.
    \end{equation}
    Here $\pi : \hat{U}^{(m_j)} \to \hat{Q}^{(r_{j})}$ denotes the projection defined by $\pi^\ast(y^i) = x^i$ for all $i \leq r$, where $\{y^i\}_{i = 1}^r$ are coordinates on $\hat{Q}^{(r_{j})}$ as in Example \ref{example_prototype_strongly_integral_submanifold}. Note that $\partial_{x^i} \sim_\pi \partial_{y^i}$ for all $i \leq r$.
    
    Denote $\chi \coloneq \pi \circ \phi \circ \rest{\iota}{V}$ and let us argue that $\chi$ is a local diffeomorphism at $q_0 := \uline{\iota}^{-1}(p_0)$. Since the tangent spaces $T_{q_0}\calR$ and $T_{\uline{\chi}(q_0)}\hat{Q}^{(r_{j})}$ have the same dimension $(r_{-j})$, it is enough to show that $T_{q_0}\chi$ is surjective. 
    As $\{x^i\}_{i = 1}^m$ are flat coordinates for $\scrD$, there is, for every $i \in \{1, \dotsc, r\}$, a vector field $Y_i \in \scrX_\calR(V)$ such that $Y_i \sim_{\iota} \partial_{x^i}$. 
    But this implies $Y_i \sim_{\chi}\partial_{y^i}$, hence in particular $(T_{q_0}\chi)(\rest{Y_i}{q_0}) = \partial_{y^i}|_{\uline{\chi}(q_0)}$. This shows surjectivity of $T_{q_0}\chi$, hence $\chi$ is a local diffeomorphism at $q_0$. 
    Similarly, $\chi' \coloneq \pi \circ \phi \circ \iota'|_{V'}$ is a local diffeomorphism at $q'_0 := \uline{\iota'}^{-1}(p_0)$. 
    
    By the $\Z$-graded inverse function theorem \cite[Theorem 4.30]{GTGM}, there exist $W \in \op_{q_0}(V)$, $W' \in \op_{q'_0}(V')$ such that $\uline{\chi}(W)$, $\uline{\chi'}(W')$ are open subsets of $\hat{Q}$, and both $\chi|_{W}$ and $\chi'|_{W'}$ are diffeomorphisms onto their image. Since $\uline{\chi}(W) \cap \uline{\chi'}(W')$ is an open subset of $\hat{Q}$ containing the point $\uline{\chi}(q_0) = \uline{\chi'}(q'_0)$, we may restrict $W$ and $W'$ if necessary, to assume that $\uline{\chi}(W) = \uline{\chi'}(W')$. Without loss of generality, let us also assume that $W$ and $W'$ are connected (we will use this at the end of the proof). We define $\vartheta : \calR'|_{W'} \to \calR|_{W}$ as the composite diffeomorphism
    \begin{equation}
        \begin{tikzcd}
            \calR'|_{W'} 
            \arrow[rr, bend right = 25, "{\vartheta}"]
            \arrow[r, "{\chi'|_{W'}}"]
                &[2em] \hat{Q}^{(r_{j})}|_{\uline{\chi}(W)}
                    \arrow[r, "{\chi|_{W}^{-1}}"]
                    &[2em] \calR|_{W}
        \end{tikzcd}.
    \end{equation}
    It remains to be argued that $\iota|_{W} \circ \vartheta = \iota'|_{W'}$. The proof will then follow from Lemma \ref{lma_similarly_immersed} and the fact that $p_0 \in \uline{\iota}(R) = \uline{\iota'}(R')$ was arbitrary. 
    
    Let us return to the vector fields $Y_i \in \scrX_{\calR}(V)$ which are $\chi$-related to $\partial_{y^i}$. Since $\chi|_{W}$ is a diffeomorphism, it follows that $\{Y_{i}|_{W}\}_{i=1}^r$ actually form a frame for $\scrX_{\calR}(W)$. Similarly, there are $Y'_i \in \scrX_{\calR'}(V')$ such that $Y'_i \sim_{\chi'}\partial_{y^i}$ and $\{Y'_i|_{W'}\}_{i=1}^r$ is a frame for $\scrX_{\calR'}(W')$. It follows that $Y'_i|_{W'} \sim_{\vartheta} Y_i|_{W}$, hence there is both
    \begin{equation}
       Y'_i|_{W'} \sim_{\iota'|_{W'}} \partial_{x^i}, \qquad \text{and} \qquad  Y'_i|_{W'} \sim_{\iota|_{W} \circ \vartheta} \partial_{x^i},
    \end{equation}
    for all $i \in \{1, \dotsc, r\}$. Consequently, for any $i \in \{1, \dotsc, r\}$ and $j \in \{1, \dotsc, m\}$ there is
    \begin{equation}
        Y'_i|_{W'} \left((\iota|_{W}\circ \vartheta)^\ast(x^j) - (\iota'|_{W'})^\ast(x^j) \right) = (\iota|_{W}\circ \vartheta)^\ast(\partial_{x^i}(x^j)) - (\iota'|_{W'})^\ast(\partial_{x^i}(x^j)) = \delta_i^j - \delta_i^j = 0.
    \end{equation}
    Since $\{Y'_i|_{W'}\}_{i=1}^r$ is a frame for $\scrX_{\calR'}(W')$, from Lemma \ref{lemma_frame_on_function_zero} it follows that
    \begin{equation}
        (\iota|_{W}\circ \vartheta)^\ast(x^j) - (\iota'|_{W'})^\ast(x^j) = c^{j},
    \end{equation}
    for every $j \in \{1, \dotsc, m\}$, where $c^j = 0$ if $|x^j| \neq 0$ and $c^{j}$ is a locally constant function if $|x^j| = 0$. In this case, we use the fact that $W'$ is connected and $c^{j}(q'_0) = 0$ to argue that $c^j = 0$ for all $j \in \{1, \dotsc, m\}$. Since $\{x^j\}_{j=1}^m$ are global coordinates on $\calM|_{U}$, necessarily $\iota|_{W}\circ \vartheta = \iota'|_{W'}$, as was to be shown.
    \end{proof}

In Corollary \ref{cor_involutive_integrable} we have seen that involutive distributions are integrable. Later in this section (Proposition \ref{prop_underlying_distribution}) we will see that an involutive distribution $\scrD$ on $\calM$ gives rise to an ``underlying'' distribution $D$ on $M$, which is also involutive and hence integrable. Our aim in the Global Frobenius theorem (Theorem \ref{thm_global_frob}) is to find a graded structure on the leaves of the foliation of $M$ given by this underlying distribution $D$. To this end, we now show the local existence of integral submanifolds $(\calR, \iota)$ of $\scrD$, such that $R$ is a subset of $M$ and $\uline{\iota}$ is the inclusion map of $R$ into $M$.

\begin{lemma}\label{lemma_graded_smooth_structure_from_domain}
Let $(n_j)$ be a sequence of non-negative integers, with finitely many nonzero entries. Let $U$ be a smooth manifold, $\hat{U} \subseteq \R^{n_0}$ and $F : \hat{U} \to U$ a diffeomorphism. Then there exists a sheaf $\cifty_\calU$ on $U$ making $\calU := (U, \cifty_{\calU})$ into a $\Z$-manifold, such that there is a diffeomorphism $\varphi : \calU \to \hat{U}^{(n_j)} $ with $\uline{\varphi}=F^{-1}$, and the original smooth structure on $U$ is the same as the one of the underlying smooth manifold of $\calU$.
\end{lemma}
\begin{proof}
    This is a very special case of one of the gluing theorems for $\Z$-manifolds, \cite[Proposition 3.33]{GTGM}. In essence, we set $\cifty_{\calU} := F_\ast(\cifty_{(n_j)})$, choose $\uline{\varphi} := F^{-1}$, $\varphi^\ast := \id$ and view $\varphi$ as a global chart for $\calU$.
\end{proof}

\begin{prop}\label{prop_local_strongly_integral_submanifolds}
Let $\scrD$ be an involutive distribution on $\calM$. Then for every point $p \in M$ there exists an integral submanifold $(\calR, \iota)$ for $\scrD$ such that $p \in R \subseteq M$ and $\uline{\iota}$ is the inclusion map.
\end{prop}
\begin{proof}
    Let $p \in M$, and consider the integral submanifold $(Q^{(r_{j})}, \iota)$ from Example \ref{example_prototype_strongly_integral_submanifold}, along with the notation set therein. Recall the situation:
    \begin{equation}
        \begin{tikzcd}
            \rest{\calM}{U} \arrow[r, "{\phi}"]
            & \hat{U}^{(m_j)}
            & \hat{Q}^{(r_{j})} \arrow[l, "{\sigma}"']
        \end{tikzcd},
    \end{equation}
    where recall that $\hat{U}$ is a cube centered around $\mathbf{0} \in \R^{m_0}$ and $\hat{Q} \subseteq \R^{r_0}$ arises as the projection of $\hat{U}$ to first $r_0$ entries. Furthermore, $\uline{\sigma}(\hat{Q})$ is the $r_0$-dimensional slice of $\hat{U}$ given by $x^{r_0+1} = \cdots = x^m = 0$, hence an embedded submanifold of $\hat{U}$. As $\uline{\phi} : U \to \hat{U}$ is a diffeomorphism, $R := \uline{\phi}^{-1}(\uline{\sigma}(\hat{Q}))$ is an embedded submanifold of $U$, hence an immersed submanifold of $M$. 
    In addition, $F := \uline{\phi}^{-1} \circ \uline{\sigma} : \hat{Q} \to R \subseteq M$ is a diffeomorphism, whose inverse is $F^{-1} = \uline{\pi} \circ \uline{\phi}|_{R}$ where $\pi : \hat{U}^{(m_j)} \to \hat{Q}^{(r_j)}$ is the projection, as in the proof of Proposition \ref{prop_similarly_immersed_strongly_integral}.
    
    By Lemma \ref{lemma_graded_smooth_structure_from_domain} there exists a graded manifold $\calR = (R, \cifty_\calR)$ and a graded diffeomorphism $\varphi : \calR \to \hat{Q}^{(r_j)}$, such that $\uline{\varphi} = F^{-1} = \uline{\pi} \circ \uline{\phi}|_{R}$. By denoting $\iota := \phi^{-1} \circ \sigma \circ \varphi$ we find that $(\calR, \iota)$ is an immersed submanifold of $\calM$ such that $\uline{\iota}$ is the inclusion map. Furthermore, as $(\hat{Q}^{(r_{j})}, \phi^{-1} \circ \sigma)$ is an integral submanifold for $\scrD$, and by definition
    \begin{equation}
        \begin{tikzcd}
        {}
        	&{\rest{\calM}{U}}
        		&{} \\
        \calR
        \arrow[ru, "{\iota}"]
        \arrow[rr, "{\varphi}"]
        	&{}
        		&{\hat{Q}^{(r_{j})}}
        		\arrow[lu, "{\phi^{-1} \circ \sigma}"']
        \end{tikzcd}
\end{equation}
commutes, $(\calR, \iota)$ is also an integral submanifold for $\scrD$, see Lemma \ref{lemma_strongly_weakly_integral}.
\end{proof}

\begin{prop}\label{prop_underlying_distribution}
Let $\scrD$ be a distribution on $\calM$. Then the assignment $U \mapsto D(U)$ where $D(U) \coloneq \{\uline{X}  \mid  X \in \scrD(U), \ |X| = 0\}$, for any $U \in \op(M)$, is a distribution on $M$, called the underlying distribution of $\scrD$. Moreover, if $(\calR, \iota)$ is a fiberwise integral submanifold for $\scrD$, then $(R, \uline{\iota})$ is an integral submanifold for $D$. Furthermore, if $\scrD$ is involutive, then so is $D$.
\end{prop}
\begin{proof}
    Let $\imath_R : R \to \calR$ and $\imath_M : M \to \calM$ denote the canonical embeddings of underlying manifolds.
    Let $(r_{-j})$ be the rank of $\scrD$ and let $\{X_i\}_{i = 1}^r$ be a local frame for $\scrD$, indexed so that $\{X_i\}_{i = 1}^{r_0}$ are all degree zero. Then $\{\uline{X_i}\}_{i = 1}^{r_0}$ locally generate $D$, see equation (\ref{eq_cifty_comb_related_VFs}) together with the fact that $\imath_M^\ast$ is surjective. Furthermore, as the tangent map $T_p\imath_M$ is injective at every point $p \in M$, the collection $\{\uline{X_i}\}_{i = 1}^{r_0}$ is a local frame for $D$. Consequently, $D$ is a rank $r_0$ distribution on the smooth manifold $M$.

    Let $(\calR, \iota)$ be a fiberwise integral submanifold for $\scrD$, and let us draw the commutative diagram
    \begin{equation}\label{eq_underlying_distr_diag}
        \begin{tikzcd}
            \calR
            \arrow[r, "{\iota}"]
                & \calM \\
            R
            \arrow[u, "{\imath_R}"]
            \arrow[r, "{\uline{\iota}}"]
                & M
                \arrow[u, "{\imath_M}"]
        \end{tikzcd}.
    \end{equation}
    It follows that $(R, \uline{\iota})$ is an integral submanifold for $D$. Indeed, consider some tangent vector $v \in T_p R$ for some $p \in R$. From the commutativity of (\ref{eq_underlying_distr_diag}) we see that $(T_{p}(\imath_M \circ \uline{\iota}))(v) \in \scrD_{\uline{\iota}(p)}$, hence $(T_p\uline{\iota})(v) \in D_{(\uline{\iota}(p))}$ by injectivity of $T_{\uline{\iota}(p)}\imath_M$. By injectivity of $T_p \uline{\iota}$ and the fact that the dimension of $R$ equals the rank of $D$, it follows that $(R, \uline{\iota})$ is an integral submanifold for $D$.

    Finally, if $\scrD$ is involutive, then consider some $X,Y \in \scrD(M)$. By definition of $D$, there is $\uline{[X,Y]} \in D(M)$. By definition of the underlying vector field, $\uline{[X,Y]} \sim_{\imath_M} [X,Y]$. But there is also $[\uline{X}, \uline{Y}] \sim_{\imath_M} [X,Y]$, hence by Proposition \ref{prop_injective_immersions} there holds $[\uline{X}, \uline{Y}] = \uline{[X,Y]} \in D(M)$. This shows $D$ to be involutive.
\end{proof}

\begin{lemma}\label{lemma_intersection_of_integral_submanifolds}
    Let $M$ be a smooth manifold, $D$ a smooth involutive distribution on $M$ and $R,R' \subseteq M$ integral submanifolds for $D$. Then $R \cap R'$ is an open subset of $R$ (in the topology of $R$).
\end{lemma}
\begin{proof}
    Let $r$ be the rank of $D$. Consider any point $p \in R \cap R'$ and some $U \in \op_p(M)$ such that $(U, \phi)$ is a chart for $M$ such that $\{x^i\}_{i = 1}^{\dim M}$ are flat coordinates for $D$ and $\phi(U)$ is a cube centered at $\phi(p) = 0$. It is a classical result, see e.g. \cite[Proposition 19.16]{ItSM} that (the image under $\phi$ of) both $R \cap U$ and $R' \cap U$ is an at most countable union of open subsets of parallel $r$-dimensional slices of $\phi(U)$. Let $V, V'$ be the components of $\phi(R \cap U)$ and $\phi(R' \cap U)$ containing $\phi(p)$. It then follows, also from \cite[Proposition 19.16]{ItSM}, that $\phi^{-1}(V \cap V')$ is an open subset of both $R$ and $R^\prime$. As $p$ was arbitrary, this finishes the proof.
\end{proof}

\begin{prop}\label{prop_globally_strongly_integral}
Let $\scrD$ be an involutive distribution on $\calM$, let $D$ be its underlying distribution and let $R \subseteq M$ be an integral submanifold for $D$. Then there is an integral submanifold $(\calR, \iota)$ for $\scrD$ with $R$ as its underlying smooth manifold, such that $\uline{\iota}$ is the inclusion map.
\end{prop}
\begin{proof}
From Proposition \ref{prop_local_strongly_integral_submanifolds} we know that for any point $p \in R$ there exists an integral submanifold $(\calR_p, \iota_p)$ where $p \in R_p \subseteq M$ and $\uline{\iota_p}$ is the inclusion map of $R_p$ into $M$. By further restricting them if necessary, let each $R_p$ be connected. From Proposition \ref{prop_underlying_distribution} we know that each $R_p \subseteq M$ is an integral submanifold for $D$, and by Lemma \ref{lemma_intersection_of_integral_submanifolds} we may, by restricting $R_p$ if necessary, assume that $R_p \in \op(R)$. Indeed, we first use the lemma to find that $R_p \cap R$ is an open subset of $R_p$ and so we may restrict the integral submanifold $\calR_p$ to it. Then we use the same lemma to see that (thus restricted) $R_p$ is an open subset of $R$. Thus $\{R_p\}_{p \in R}$ is an open cover of $R$.

The idea now is to use one of the collation theorems for $\Z$-manifolds, namely \cite[Proposition 3.31]{GTGM}. For every $p,q \in R$, Proposition \ref{prop_similarly_immersed_strongly_integral} ensures that the integral submanifolds $(\rest{\calR_p}{R_{pq}}, \rest{\iota_p}{R_{pq}})$ and $(\rest{\calR_q}{R_{pq}}, \rest{\iota_q}{R_{pq}})$ are equivalently immersed, where $R_{pq} := R_p \cap R_q$. In other words, there is a diffeomorphism $\vartheta_{pq} : \rest{\calR_q}{R_{pq}} \to \rest{\calR_p}{R_{pq}}$ such that $\uline{\vartheta_{pq}} = \id$ and $\rest{\iota_p}{R_{pq}} \circ \vartheta_{pq} = \rest{\iota_q}{R_{pq}}$. Moreover, for any $p,q,w \in R$ we can draw
\begin{equation}
\begin{tikzcd}[column sep = tiny]
&{}
	&{\calM}
		&{} \\
&{}
	&{}
		&{} \\
&{\rest{\calR_p}{R_{pqw}}}
\arrow[rr, "{\vartheta_{wp}}", pos = 0.3]
\arrow[rd, "{\vartheta_{qp}}"']
\arrow[ruu, "{\iota_p}"]
	&{}
		&{\rest{\calR_w}{R_{pqw}}}
		\arrow[luu, "{\iota_w}"'] \\
&{}
	&{\rest{\calR_q}{R_{pqw}}}
	\arrow[uuu, "{\iota_q}"', crossing over, pos = 0.7]
	\arrow[ru, "{\vartheta_{wq}}"']
		&{} \\
\end{tikzcd},
\end{equation}
where we omit explicitly writing the restrictions of morphisms. Here all the faces of the diagram are known to commute, except for the bottom one. Consequently, there is $\iota_w \circ \vartheta_{wq} \circ \vartheta_{qp} = \iota_w \circ \vartheta_{wp}$, again implicitly restricted. But as $\iota_w$ is a monomorphism by Proposition \ref{prop_injective_immersions}, there must be
\begin{equation}
    \vartheta_{wq} \circ \vartheta_{qp} = \vartheta_{wp},
\end{equation}
i.e. the cocycle condition is satisfied. As a result, \cite[Proposition 3.31]{GTGM} outputs a $\Z$-manifold $\calR$ along with diffeomorphisms $\lambda_p : \rest{\calR}{R_p} \to \calR_p$ such that $\uline{\lambda_p} = \id$ and $\lambda_q  =  \vartheta_{qp} \circ \lambda_p$, when restricted to $R_{pq}$. 

Since $(\calR_p, \iota_p)$ is an integral submanifold for $\scrD$ and $\lambda_p$ is a diffeomorphism, $(\rest{\calR}{R_p}, \iota_p \circ \lambda_p)$ is also an integral submanifold for $\scrD$ for every $p \in R$. Let us draw the diagram
\begin{equation}
    \begin{tikzcd}[column sep = tiny]
    {}
        &{\calM}
            &{}\\
    {\rest{\calR_p}{R_{pq}}}
    \arrow[ru, "{\iota_p}"]
    \arrow[rr, "{\vartheta_{qp}}"]
        &{}
            &{\rest{\calR_q}{R_{pq}}}
            \arrow[lu, "{\iota_q}"']\\
    {}
        &{\rest{\calR}{R_{pq}}}
        \arrow[lu, "{\lambda_p}"']
        \arrow[ru, "{\lambda_q}"]
            &{}\\
    \end{tikzcd},
\end{equation}
which commutes, since both its component triangles commute. As a result, the collection of graded smooth maps $\{\iota_p \circ \lambda_p\}_{p \in R}$ agree on overlaps, hence they glue together a graded smooth map
\begin{equation}
    \iota : \calR \to \calM,
\end{equation}
such that $\rest{\iota}{R_p} = \iota_p \circ \lambda_p$ for every $p \in R$. The map $\iota$ is easily seen to be an immersion such that $\uline{\iota}$ is the inclusion map. It follows from Lemma \ref{lemma_locally_strongly_integral} that $(\calR, \iota)$ is an integral submanifold for $\scrD$.
\end{proof}

We now combine all that we learned so far in the Global Frobenius theorem.

\begin{theorem}[Global Frobenius]\label{thm_global_frob}
Let $\scrD$ be an involutive distribution on $\calM$. Then for every leaf $R$ of the foliation given by the underlying distribution $D$ of $\scrD$ there exists a $\Z$-manifold $\calR = (R, \cifty_\calR)$ and an immersion $\iota : \calR \to \calM$, such that $\uline{\iota}$ is the inclusion map and $(\calR, \iota)$ is an integral submanifold for $\scrD$.

Furthermore, the following maximality and uniqueness conditions hold: for any other connected integral submanifold $(\calR', \iota')$ for $\scrD$ such that $\uline{\iota'}(R') \cap R \neq \emptyset$ there is $\uline{\iota'}(R') \in \op(R)$, and $(\calR', \iota')$ and $(\rest{\calR}{\uline{\iota'}(R')}, \rest{\iota}{\uline{\iota'}(R')})$ are equivalently immersed.
\end{theorem}

\begin{proof}
    From Proposition $\ref{prop_underlying_distribution}$ we know that the underlying distribution $D$ is involutive. From the global Frobenius theorem for smooth manifolds we obtain a foliation of $M$ by integral submanifolds of $D$. Let $R$ be a leaf of this foliation. By Proposition \ref{prop_globally_strongly_integral} there exists an integral submanifold $\calR = (R, \iota)$ for $\scrD$ such that $\uline{\iota}$ is the inclusion map. 
    
    Let $\calR' = (R', \iota')$ be some connected integral submanifold for $\scrD$ such that $\uline{\iota'}(R') \cap R \neq \emptyset$. Since $R'$ is connected, so is $\uline{\iota'}(R')$ and so there must be $\uline{\iota'}(R') \subseteq R$. Because $\uline{\iota'}(R')$ is also an integral submanifold for $D$, $\uline{\iota'}(R')$ must be open in the topology of $R$ by Lemma \ref{lemma_intersection_of_integral_submanifolds}. It then follows from Proposition \ref{prop_similarly_immersed_strongly_integral} that $(\calR', \iota')$ and $(\calR|_{\uline{\iota'}(R')}, \iota|_{\uline{\iota'}(R')})$ are equivalently immersed.
\end{proof}

We finish the section with an important example of an involutive distribution: the kernel of $T\varphi$, where $\varphi$ is a submersion.

\begin{example}(Submersion Level Sets)\label{example_kernel_of_sumbersion}
    Consider a submersion $\varphi : \calM \to \calN$. One can define $\ker(T\varphi)\subseteq \scrX_\calM$ as the kernel of the induced morphism of sheaves of $\cifty_\calM$-modules $\hat{T}\varphi : \scrX_\calM \to \varphi^\star(\scrX_\calN)$, see (\ref{diag_pullback_tangent_bundle}) for the notation. Since $\varphi$ is a submersion, $T\varphi$ is fiberwise surjective and so $\ker(T\varphi)$ is a subbundle of $T\calM$, see \cite[Proposition 3.14]{3GVB}. It then follows from Corollary \ref{cor_related_VFs} that $X \in [\ker(T\varphi)](M)$ if and only if $X \sim_\varphi 0$, which in turn immediately implies that $\ker(T\varphi)$ is an involutive distribution.

    We also know that for every point $q \in \uline{\varphi}(M)$ there exists the \textbf{level set submanifold}, i.e. a closed embedded submanifold $(\varphi^{-1}(q), \iota_q)$ of $\calM$ with $\uline{\varphi}^{-1}(q)$ as its underlying smooth manifold, which is unique inasmuch as it is a categorical pullback
    \begin{equation}\label{diag_level_set}
        \begin{tikzcd}
            \varphi^{-1}(q)
            \arrow[r, "{}"]
            \arrow[d, "{\iota_q}"]
                &\{\ast\}
                \arrow[d, "{q}"] \\
            \calM
            \arrow[r, "{\varphi}"]
                &\calN
        \end{tikzcd},
    \end{equation}
    in the category of $\Z$-manifolds, per \cite[Theorem 7.43]{GTGM}. Here $q : \{\ast\} \to \calN$ denotes the one-point embedding of $q$ into $\calN$, and the unmarked arrow is the unique arrow to a terminal object. Let us argue that $(\varphi^{-1}(q), \iota_q)$ is an integral submanifold for $\ker(T\varphi)$ for any $q \in \uline{\varphi}(M)$. 

    Fix some $q \in \uline{\varphi}(M)$. Note that the underlying distribution of $\ker(T\varphi)$ is $\ker(T \uline{\varphi}) \subseteq \scrX_M$. Using the classical result, we know that $\uline{\varphi}^{-1}(q)$ is an integral submanifold for $\ker(T\uline{\varphi})$. Since $\ker(T\varphi)$ is involutive, by Proposition \ref{prop_globally_strongly_integral} there exists an integral submanifold $(\calR, \iota)$ of $\ker(T\varphi)$, such that $R = \uline{\varphi}^{-1}(q)$ and $\uline{\iota}$ is the inclusion map. 
    
    Let us argue that $(\calR, \iota)$ also fits in the place of $(\varphi^{-1}(q), \iota_q)$ in Diagram (\ref{diag_level_set}). One must show that $\varphi \circ \iota = q \circ t$ where we denoted as $t$ the arrow to the terminal object $\calR \to \{\ast\}$. On the level of underlying smooth maps this is true, as for any $p \in R \equiv \uline{\varphi}^{-1}(q)$ there is $(\uline{\varphi} \circ \uline{\iota}')(p) = q = (q \circ t)(p)$. On the level of pullbacks we have, for any $U \in \op_q(N)$ and $f \in \cifty_\calN(U)$, that $(q \circ t)^\ast(f) = \uline{f}(q)$. We must show that also $(\iota^\ast \circ \varphi^\ast)(f) = \uline{f}(q)$. 
    From Proposition \ref{prop_equivalently_integral} we know that $\im(\hat{T}\iota) = \iota^\star\ker(T\varphi)$. It follows that the induced map $\hat{T}(\varphi \circ \iota) : T \calR \to (\varphi \circ \iota)^\star T \calN$ factors through $\iota^\star \ker(T \varphi)$, and hence must be zero. Thus, for any $Y \in \scrX_\calR(R)$ there is $Y \sim_{\varphi \circ \iota} 0$ by Corollary \ref{cor_related_VFs}, i.e.
    \begin{equation}
        Y\left[(\iota^\ast \circ \varphi^\ast)(f)\right] = 0.
    \end{equation}
    From Lemma \ref{lemma_frame_on_function_zero} it follows that $(\iota^\ast \circ \varphi^\ast)(f)$ is a locally constant function. Since $\uline{(\iota^\ast \circ \varphi^\ast)(f)} = \uline{f}\circ\uline{\varphi} \circ \uline{\iota}$, it follows that $(\iota^\ast \circ \varphi^\ast)(f) = \uline{f}(q)$. We have shown that $\varphi \circ \iota = q \circ t$.

    As a result, there exists $\psi : \calR \to \varphi^{-1}(q)$ such that, in particular, $\iota_q \circ \psi = \iota$. Consequently, $T \iota_q \circ T\psi = T \iota$, which means that $T\psi$ is fiberwise bijective, and so $\psi$ is a local diffeomorphism at every point $p \in R$. Since $\uline{\psi} = \id$, we conclude that $\psi$ is a diffeomorphism such that $\iota_q \circ \psi = \iota$. In other words, $(\calR, \iota)$ and $(\varphi^{-1}(q), \iota_q)$ are equivalently immersed. By Lemma \ref{lemma_strongly_weakly_integral}, $(\varphi^{-1}(q), \iota_q)$ is an integral submanifold for $\ker(T\varphi)$.
\end{example}

\section{Pointwise Involutivity and the Converse Implication}\label{section_6}

We have seen that an involutive distribution is integrable, and in Theorem \ref{thm_global_frob} we extended this knowledge with maximality and uniqueness, in a suitable sense. We begin this final section with two examples of distributions which are integrable, but not involutive. We then search for a weaker notion of involutivity, which does follow from integrability.

\begin{example}\label{example_integrable_noninvolutive_distr}
Consider the one-point manifold $\calM = \{\ast\}^{(m_j)}$ with coordinates $\xi^1, \xi^2,$ of degree 1 and $ \eta$ of degree 3, and $\calR = \{\star\}^{(s_j)}$ with coordinate $\theta$ of degree 1. Let $\iota : \calR \to \calM$ be given by $\iota^\ast(\xi^1) = \theta, \ \iota^\ast(\xi^2) = 0$ and $\iota^\ast(\eta) = 0$. Consider a distribution $\scrD$ on $\calM$ generated by the vector field
\begin{equation}
X  = \partial_{\xi^1} + \xi^1 \xi^2 \, \partial_{\eta}.
\end{equation}
With direct calculation we find that $[X,X] = 2\xi^2\partial_{\eta}$, ergo $\scrD$ is not involutive. Let us show that $(\calR, \iota)$ is an integral submanifold for $\scrD$. We do this by verifying that $X$ is tangent to $\calR$ using Proposition \ref{prop_embedded_tangent_criterion} and Lemma \ref{lemma_equivalent_strongly_integral}. We go through all possible degrees of graded functions $f \in \cifty_\calM(\ast)$.
\begin{itemize}
\item $|f| = 0$. Any such has $Xf = 0$ for degree reasons.
\item $|f| = 1$. Then $f = \alpha \xi_1 + \beta \xi_2$ for some $\alpha, \beta \in \R$. We have $\iota^\ast(f) = \alpha \theta$ which is zero iff $\alpha = 0$. For any such $f$ we have $Xf = \alpha = 0$, and so $\iota^\ast(Xf) = 0$.
\item $|f| = 2$. Then $f = \alpha \xi_1 \xi_2$ for some $\alpha \in \R$. Always $\iota^\ast(f) = 0$ and $Xf = \alpha \xi_2$, hence $\iota^\ast(Xf) = 0$.
\item $|f| \in \{3,4,5\}$. For degree reasons always $\iota^\ast(f) = 0$ and also $\iota^\ast(Xf) = 0$.
\end{itemize}
This shows that $(\calR, \iota)$ is indeed an integral submanifold for $\scrD$.
\end{example}

In the above example involutivity fails due to a non-homological generator of the distribution. This is not the case for the next example.

\begin{example}\label{example_integrable_noninvolutive_distr2}
Let $\calM = (\R^2)^{(m_j)}$ with graded coordinates $\xi$ of degree -1 and $\eta$ of degree 1, and standard coordinates $x,y$ on $\R^2$. Consider two global vector fields $X,Y$ given by
\begin{equation}
X := \partial_{x} + \xi \partial_{\xi}, \qquad Y:= \partial_{\eta} + \xi\partial_{y},
\end{equation}
of degrees $|X| = 0$ and $|Y| = -1$. We see that at every point $p \in M \equiv \R^2$, the tangent vectors $X|_p = \partial_{x}|_p, Y|_{p} = \partial_{\eta}|_{p}$ are linearly independent, hence we may define a distribution $\scrD$ on $\calM$ as their $\cifty_\calM$-span. The vector fields $X,Y$ then form a global frame for $\scrD$. Direct computation yields
\begin{equation}
[X,Y] = \xi \partial_{y} \ \notin \scrD(M),
\end{equation}
so $\scrD$ is not involutive. Let us also note that $[Y,Y] \equiv 2Y^2 = 0$, i.e. that $Y$ is homological. Now consider $\calR = \R^{(r_j)}$ with coordinates $t$ of degree zero and $\theta$ of degree one. Fix some $y_0 \in \R$ and let $\iota: \calR \to \calM$ be defined as $\uline{\iota}(q) = (q,y_0)$ and
\begin{align}
\iota^\ast(x) := t, \quad \iota^\ast(y) := y_0, \quad \iota^\ast(\xi) := 0, \quad \iota^\ast(\eta) := \theta.
\end{align}
First, we see that $(T_q\iota)(\partial_{t}|_{q}) = \partial_{x}|_{(q, y_0)}$ and $(T_q\iota)(\partial_{\theta}|_{q}) = \partial_{\eta}|_{(q,y_0)}$, which means that $T_q\iota$ is injective for every $q \in \R$ and so $\iota$ is an immersion. In fact, since $\uline{\iota}$ is just the insertion of $\R$ as the constant slice $y = y_0$ in $\R^2$, it is an embedding.

We ask whether $(\calR, \iota)$ is an integral submanifold for $\scrD$ for any value of $y_0$. Clearly the ranks of $\scrX_\calR$ and $\scrD$ agree, so let us verify that both $X$ and $Y$ are tangent to $\calR$. We will again use the criterion in Proposition \ref{prop_embedded_tangent_criterion}, degree by degree:
\begin{itemize}
\item $|f| = -1$. For degree reasons any such $f$ has $\iota^\ast(f) = 0$, $\iota^\ast(Xf)=0$ and $\iota^\ast(Yf) = 0$.
\item $|f| = 0$. Then $f = \alpha(x,y) + \beta(x,y)\xi\eta$ for some $\alpha, \beta \in \cifty(\R^2)$. Note that we defined $\iota^\ast(x) = x \circ \uline{\iota}$ and $\iota^\ast(y) = y \circ \uline{\iota}$, i.e. with no purely graded part, we have $\iota^\ast(f) =  \iota^\ast(\alpha) = \alpha \circ \uline{\iota} \equiv \alpha(\cdot, y_0)$, and we see that $\iota^\ast(f) = 0$ iff $\alpha(t, y_0) = 0,$ for every $t \in \R.$ In such case we have
\begin{equation}
Xf = \dd{\alpha}{x} + \dd{\beta}{x}\xi \eta + \beta \xi \eta \implies \iota^\ast(Xf)  = \partial_1\alpha(t,y_0) = 0.
\end{equation}
Finally, $\iota^\ast(Yf) = 0$ for degree reasons.
\item $|f| = 1$. Then $f = \alpha(x,y)\eta$ for some $\alpha \in \cifty(\R^2)$. We have $\iota^\ast(f) = 0$ if and only if $\alpha(t,y_0) = 0$ for all $t \in \R$. For such $f$ there is $Xf = (\partial_x\alpha)\eta$ hence $\iota^\ast(Xf) = (\partial_1\alpha)(t,y_0)\theta = 0$ for all $t \in \R$. Similarly, $Yf = \alpha(x,y) + (\partial_2\alpha)(x,y)\xi\eta$ and so $\iota^\ast(Yf) = \alpha(t,y_0) + 0 = 0$ for all $t \in \R$. 
\end{itemize}
All this means that $(\calR, \iota)$ is indeed an integral manifold for $\scrD$ for any value of $y_0$.
\end{example}

In both examples we saw that the distribution was not involutive, but only up to a vector field which is zero at every point. This is a general result, which we now formalize.

\begin{definition}\label{def_pointwise_involutive}
    We call a distribution $\scrD$ on $\calM$ \textbf{pointwise involutive} if for every $U \in \op(M)$ and all $X,X' \in \scrD(U)$ there is $[X,X']|_{p} \in \scrD_{p}$ for every $p \in U$.
\end{definition}

In exactly the same way as in Lemma \ref{lemma_local_to_global_involutivity}, one can show that pointwise involutivity can be verified globally, or locally using a frame.

\begin{theorem}\label{thm_on_foliated_distributions}
Let $\scrD$ be an integrable distribution on $\calM$. Then $\scrD$ is pointwise involutive.
\end{theorem}
\begin{proof}
    Let us fix some $p \in M$. Since $\scrD$ is a distribution, there exists $U \in \op_p(M)$ and a frame $\{X_i\}_{i = 1}^m$ for $\scrX_\calM(U)$ such that $\{X_i\}_{i = 1}^{r}$ is a frame for $\scrD(U)$. Consider some $i,j \in \{1, \dotsc, r\}, \ i \leq j$ and write
    \begin{equation}\label{eq_pointwise_frob_1}
        [X_i, X_j] = \underbrace{\sum_{k = 1}^{r}f_{k}^{(i,j)} X_k}_{=: Y^{(i,j)}} + \underbrace{\sum_{\ell = r+1}^{m}f_\ell^{(i,j)} X_\ell}_{ =: Z^{(i,j)}},
    \end{equation}
    for some graded functions $f_k^{(i,j)}, f_{\ell}^{(i,j)} \in \cifty_\calM(U)$. Clearly $Y^{(i,j)} \in \scrD(U)$. Let us show that $Z^{(i,j)}|_q = 0$ for any $q \in U$. By assumption, any $q \in U$ is contained within the image $\uline{\iota}(R)$ of some integral submanifold $(\calR, \iota)$ for $\scrD$. Since for every $k \leq r$ the vector field $X_k$ is tangent to $\calR$, then so are $[X_i, X_j]$ and $Z^{(i,j)}$ for every $i \leq j \leq r$. Consequently $Z^{(i,j)}$ must also be tangent to $\calR$ and in particular $Z^{(i,j)}|_{q} \in \scrD_{q}$. Recall that $\{X_i|_{q}\}_{i = 1}^m$ is a basis of $T_q\calM$, and the subset $\{X_i|_{q}\}_{i = 1}^r$ is a basis of $\scrD_{q}$. However, in (\ref{eq_pointwise_frob_1}) we see that $Z^{(i,j)}|_{q}$ is a linear combination of $X_\ell|_{q} $ for $\ell > r$, hence necessarily $Z^{(i,j)}|_{q} = 0$.

    Next, consider some arbitrary $X,X' \in \scrD(U)$. From the Leibniz rule it follows that we can write
    \begin{equation}
    \begin{split}
        [X,X'] &= \sum_{k = 1}^r g_k X_k + \sum_{i \leq j} g_{ij} \, [X_i, X_j] \\
        &= \underbrace{\sum_{k = 1}^r g_k X_k + \sum_{i \leq j} g_{ij}  \, Y^{(i,j)}}_{=: Y} + \underbrace{\sum_{i \leq j} g_{ij}  \, Z^{(i,j)}}_{=: Z},
    \end{split}
    \end{equation}
    for some graded functions $g_k, g_{ij} \in \cifty_\calM(U)$. We see that $Y \in \scrD(U)$ and $Z|_{q} = 0$ for every $q \in U$.

\end{proof}

\begin{remark}
     In Theorem \ref{thm_on_foliated_distributions}, the converse implication  does not hold; see Example \ref{example_pointwise_involutive_nonintegrable} for a pointwise involutive distribution which admits no integral submanifolds. Furthermore, in the wording of the theorem, one cannot strengthen pointwise involutivity to involutivity; see Examples \ref{example_integrable_noninvolutive_distr} and \ref{example_integrable_noninvolutive_distr2}.
\end{remark}

The following example illustrates that Theorem \ref{thm_on_foliated_distributions} is indeed only an implication.

\begin{example}\label{example_pointwise_involutive_nonintegrable}
    Consider a one point manifold $\calM$ with coordinates $\xi^1, \xi^2$ of degree $-1$ and $\eta^1, \eta^2$ of degree $1$. On $\calM$ let the distribution $\scrD$ be given as a span of the following vector fields: 
    \begin{equation}
        X:= \partial_{\xi^1}, \quad Y := \partial_{\xi^2}, \quad Z := \partial_{\eta^1} + \xi^1 \eta^1 \partial_{\eta^2}
    \end{equation}
     The only non-zero commutators of the generators are
    \begin{equation}
        [X,Z] = \eta^1 \partial_{\eta^2}, \quad [Z,Z] = -2\,\xi^1\partial_{\eta^2},
    \end{equation}
    which are zero at every point. From the Leibniz rule it follows that $\scrD$ is pointwise involutive. We will show that $\scrD$ admits no integral submanifolds. Suppose that one exists and call it $(\calR, \iota)$. Then necessarily $R$ is a one-point manifold, and hence must be diffeomorphic to a graded domain. Thus, without loss of generality, let $\calR = \{\star\}^{(r_j)}$. Denote the two degree $-1$ coordinates as $\zeta^1, \zeta^2$ and the degree $1$ coordinate as $\theta$. We have 
    \begin{align}
        \iota^\ast(\xi^i) &= {A^{i}}_j \, \zeta^j + B^i \, \zeta^1 \zeta^2 \theta, \\
        \iota^\ast(\eta^i) &= C^i \, \theta,
    \end{align}
    for some ${A^i}_j, B^i, C^i \in \R$. Direct computation of the tangent map yields
    \begin{equation}
        (T_\star\iota)(\rest{\partial_{\zeta^i}}{\star}) = {A^j}_i \, \rest{\partial_{\xi^j}}{\ast}, \qquad (T_\star\iota)(\rest{\partial_{\theta}}{\star}) = C^j \rest{\partial_{\eta^j}}{\ast}.
    \end{equation}
    As $\iota$ is assumed to be an immersion, ${A^i}_j$ must form an invertible matrix and at least one of $C^1, C^2$ must be non-zero. Furthermore, as $(\calR, \iota)$ is, in particular, a fiberwise integral submanifold for $\scrD$, there must be $C^2 = 0$, hence $C^1 \neq 0$ and we have $\iota^\ast(\eta^1) = C^1 \, \theta$ and $\iota^\ast(\eta^2) = 0$. Now consider the degree zero function $f := \xi^2 \eta^2$ on $\calM$, which clearly satisfies $\iota^\ast(f) = 0$. Observe that $Zf =  - \xi^2 \xi^1 \eta^1 = \xi^1 \xi^2 \eta^1$, and so
    \begin{equation}
    \begin{split}
         \iota^\ast(Zf) &= \left({A^1}_j \, \zeta^j + B^1 \zeta^1 \zeta^2 \theta\right)\left( {A^2}_j \, \zeta^j + B^2 \zeta^1 \zeta^2 \theta\right)  C^1 \, \theta = C^1\, \theta \, {A^1}_j \, \zeta^j {A^2}_k \, \zeta^k  \\
        &= C^1 \, \left({A^1}_1 {A^2}_2 - {A^1}_2 {A^2}_1  \right)\zeta^1 \zeta^2 \, \theta = C^1 \, \det(A) \, \zeta^1 \zeta^2 \theta.
    \end{split}
    \end{equation}
    This is evidently not zero, which is in contradiction with $(\calR, \iota)$ being an embedded integral submanifold for $\scrD$. Hence $\scrD$ admits no integral submanifolds.
\end{example}

One may ask whether in the statement of Theorem \ref{thm_on_foliated_distributions} one could require every point to be contained within a fiberwise integral submanifold. The final example shows that this is not enough.

\begin{example}\label{ex_final}
    Let $\calM = \{\ast\}^{(m_j)}$ be a one point manifold with coordinates $\xi^1, \xi^2$ of degree 1 and $\eta$ of degree 2. Let $\scrD$ be a distribution generated by the vector fields
    \begin{equation}
        X := \partial_{\xi^1}, \qquad Y := \partial_{\xi^2} + \xi^1 \partial_\eta,
    \end{equation}
    where both vector fields are of degree $-1$. The distribution is not pointwise involutive, as
    \begin{equation}
        [X,Y] = \partial_\eta.
    \end{equation}
    We will show that $\scrD$ admits a fiberwise integral submanifold. Let $(\calR, \iota)$ be any immersed (thus embedded) submanifold of $\calM$. Without loss of generality let $\calR$ be the graded domain $\{\star\}^{(r_j)}$. If it is to be a fiberwise integral submanifold, it must have coordinates $\zeta^1, \zeta^2$ of degree 1. The most general form of $\iota$ is then given by
    \begin{equation}
        \iota^\ast(\xi^i) = {A^i}_j \, \zeta^j, \qquad \iota^\ast(\eta) = B\, \zeta^1 \zeta^2,
    \end{equation}
    for some ${A^i}_j, B \in \R$. The tangent map gives
    \begin{equation}   
    (T_\ast \iota)(\partial_{\zeta^i}|_{\ast}) = {A^j}_i \, \partial_{\xi^j}|_{\ast}.
    \end{equation}
    Again, since $\iota$ is an immersion, $A$ must be an invertible matrix. If this is so, then $(\calR, \iota)$ is automatically a fiberwise integral submanifold for $\scrD$.

    From Theorem \ref{thm_on_foliated_distributions} we know there cannot be any integral submanifold for $\scrD$. As a curiosity, let us verify that this is so. For contradiction suppose that $(\calR, \iota)$ is integral, and consider the degree 2 function
    \begin{equation}
        f := \eta - \frac{B}{\det(A)}\xi^1 \xi^2,
    \end{equation}
    which is constructed so that $\iota^\ast(f) = 0$. By integrality there must be $\iota^\ast(Xf) = 0$ and $\iota^\ast(Yf) = 0$. We can write
    \begin{equation}
    0 = \iota^\ast \left(Xf \right) = \iota^\ast \left( - \frac{B}{\det(A)}\xi^2 \right) = -\frac{B}{\det(A)} {A^2}_{j} \, \zeta^j.
    \end{equation}
    Note that at least one of ${A^2}_1, \ {A^2}_2$ must be non-zero, otherwise $A$ would be singular, hence $B = 0$. Next we have
    \begin{equation}
        0 = \iota^\ast \left(Yf \right) =  \iota^\ast \left( \xi^1 \right) = {A^1}_j \, \zeta^j.
    \end{equation}
    Again, at least one of ${A^1}_1, \ {A^1}_2$ must be non-zero, which yields the contradiction.
\end{example}

\addcontentsline{toc}{section}{References}

\begingroup

\hypersetup{hidelinks}
      \printbibliography

@article{3GVB,
title = {Threefold nature of graded vector bundles},
journal = {Journal of Geometry and Physics},
volume = {216},
pages = {105557},
year = {2025},
issn = {0393-0440},
doi = {https://doi.org/10.1016/j.geomphys.2025.105557},
url = {https://www.sciencedirect.com/science/article/pii/S039304402500141X},
author = {Rudolf \v{S}molka and Jan Vysok\'{y}}
}

@book{ItSM,
  title     = {Introduction to Smooth Manifolds},
  author    = {Lee, John M.},
  year      = {2012},
  edition   = {2},
  series    = {Graduate Texts in Mathematics},
  volume    = {218},
  publisher = {Springer-Verlag},
  address   = {New York},
  isbn      = {978-1-4419-9982-5},
  doi       = {10.1007/978-1-4419-9982-5}
}

@article{GTGM,
author = {Vysok\'{y}, Jan},
title = {Global theory of graded manifolds},
journal = {Reviews in Mathematical Physics},
volume = {34},
number = {10},
pages = {2250035},
year = {2022},
doi = {10.1142/S0129055X22500350},
URL = {https://doi.org/10.1142/S0129055X22500350}
}

@misc{Z2nFT,
title={{The Frobenius theorem for $\mathbb{Z}^n_2$-supermanifolds}}, 
author={Tiffany Covolo and Stephen Kwok and Norbert Poncin},
year={2016},
eprint={1608.00961},
archivePrefix={arXiv},
primaryClass={math.DG},
url={https://arxiv.org/abs/1608.00961}, 
}

@article{NFT,
  author       = {Bursztyn, Henrique and Cueca, Miquel and Mehta, Rajan Amit},
  title        = {A geometric characterization of {$\mathbb{N}$}‑graded manifolds and the Frobenius theorem},
  journal      = {Journal of Noncommutative Geometry},
  year         = {2025},
  doi          = {10.4171/JNCG/623},
  url          = {https://doi.org/10.4171/JNCG/623}
}

@misc{Zflows,
      title={Flows on Graded Manifolds}, 
      author={\v{S}molka, Rudolf and Vysok\'{y}, Jan},
      year={2026},
      eprint={2605.22910},
      eprinttype={arXiv},
      eprintclass={math.DG},
      url={https://arxiv.org/abs/2605.22910}, 
}

@book{Hartshorne1977,
  publisher = {Springer New York},
  author    = {Hartshorne, Robin},
  title     = {Algebraic Geometry},
  journal   = {Graduate Texts in Mathematics},
  year      = {1977},
  doi       = {10.1007/978-1-4757-3849-0}
}

@article{Monterde1997,
  title = {Geometric properties of involutive distributions on graded manifolds},
  volume = {8},
  ISSN = {0019-3577},
  url = {http://dx.doi.org/10.1016/S0019-3577(97)89122-7},
  DOI = {10.1016/s0019-3577(97)89122-7},
  number = {2},
  journal = {Indagationes Mathematicae},
  publisher = {Elsevier BV},
  author = {Monterde,  Juan L. and Muñoz-Masqué,  Jaime and Sánchez-Valenzuela,  Oscar A.},
  year = {1997},
  %month = {June},
  pages = {217–246}
}

@article{Bruzzo1985,
  title = {Differential equations,  Frobenius theorem and local flows on supermanifolds},
  volume = {18},
  ISSN = {1361-6447},
  url = {http://dx.doi.org/10.1088/0305-4470/18/3/017},
  DOI = {10.1088/0305-4470/18/3/017},
  number = {3},
  journal = {Journal of Physics A: Mathematical and General},
  publisher = {IOP Publishing},
  author = {Bruzzo, Ugo and Cianci,  Roberto},
  year = {1985},
  month = Feb,
  pages = {417–423}
}

@article{Deahna1840,
author = {Deahna, Feodor},
journal = {Journal für die reine und angewandte Mathematik},
language = {ger},
pages = {340-349},
title = {Ueber die Bedingungen der Integrabilität lineärer Differentialgleichungen erster Ordnung zwischen einer beliebigen Anzahl veränderlicher Größen.},
url = {http://eudml.org/doc/147105},
volume = {20},
year = {1840},
}

@article{Frobenius1877,
author = {Frobenius, Ferdinand Georg},
journal = {Journal für die reine und angewandte Mathematik},
pages = {230-315},
title = {Ueber das Pfaffsche Problem.},
url = {http://eudml.org/doc/148315},
volume = {82},
year = {1877},
}

@article{Clebsch1866,
author = {Clebsch, Alfred},
journal = {Journal für die reine und angewandte Mathematik},
language = {ger},
pages = {257-268},
title = {Ueber die simultane Integration linearer partieller Differentialgleichungen.},
url = {http://eudml.org/doc/147970},
volume = {65},
year = {1866},
}

@article{Androulidakis2009,
  title = {The holonomy groupoid of a singular foliation},
  volume = {2009},
  ISSN = {1435-5345},
  url = {http://dx.doi.org/10.1515/CRELLE.2009.001},
  DOI = {10.1515/crelle.2009.001},
  number = {626},
  journal = {Journal f\"{u}r die reine und angewandte Mathematik (Crelles Journal)},
  publisher = {Walter de Gruyter GmbH},
  author = {Androulidakis,  Iakovos and Skandalis,  Georges},
  year = {2009},
  month = Jan,
  pages = {1–37}
}

@inbook{CrainicFernandes2011,
title = "Lectures on integrability of Lie brackets",
author = "Crainic, Marius and Fernandes, Rui Loja",
year = "2011",
language = "English (US)",
volume = "17",
series = "Geom. Topol. Monogr.",
publisher = "Geom. Topol. Publ., Coventry",
pages = "1--107",
booktitle = "Lectures on Poisson geometry",

}

\endgroup

\end{document}